\documentclass[10pt,reqno]{amsart}
\usepackage{amssymb,amsmath,graphicx,amsfonts,euscript}
\usepackage{color}

\newtheorem{thm}{Theorem}[section]

\newtheorem{prop}[thm]{Proposition}
\newtheorem{define}[thm]{Definition}
\newtheorem{lemma}[thm]{Lemma}
\theoremstyle{remark}
\newtheorem{rem}{Remark}

\newcommand{\e}{\epsilon}
\newcommand{\p}{\partial}
\newcommand{\g}{\gamma}
\newcommand{\A}{\alpha}

\newcommand{\na}{\nabla}
\newcommand{\la}{\lambda}
\newcommand{\de}{\delta}

\newcommand{\R}{\mathbb{R}}
\newcommand{\Z}{\mathbb{Z}}
\newcommand{\N}{\mathbb{N}}

\newcommand{\F}{\mathcal{F}}

\newcommand{\PP}{\mathcal{P}}

\newcommand{\CC}{\mathbb{C}}
\newcommand{\les}{\lesssim}

\newcommand{\T}{\mathcal{T}}

\newcommand{\NN}{\mathcal{N}}

\newcommand{\Da}{\mathcal{W}}

\newcommand{\h}{\mathcal{H}}
\newcommand{\Y}{\mathcal{Y}}

\newcommand{\M}{\mathcal{M}}
\def\beq{\begin{equation}}
\def\eeq{\end{equation}}
\numberwithin{equation}{section}
\begin{document}
\title[Sharp maximal function  estimates  for   Hilbert transforms]{Sharp maximal function  estimates  for   Hilbert transforms along monomial curves in  higher dimensions}

\author[R. Wan]{ Renhui Wan}

\address{School of Mathematical Sciences and Institute of Mathematical Sciences, Nanjing Normal University, Nanjing 210023, People's Republic of China\\
ORCID:https://orcid.org/0009-0001-4558-7445}

\email{wrh@njnu.edu.cn}

\vskip .2in
\begin{abstract}
For any nonempty set $U\subset\R^+$, we consider the maximal operator $\h^U$ defined as $\h^Uf=\sup_{u\in U}|H^{(u)} f|$, where $H^{(u)}$ represents the Hilbert transform along the monomial curve $u\gamma(s)$. We focus on the $L^p(\mathbb{R}^d)$ operator norm of $\h^U$ for $p\in (p_\circ(d),\infty)$, where $p_\circ(d)$ is the optimal exponent known for the $L^p$ boundedness of the maximal averaging operator obtained by Ko-Lee-Oh \cite{KLO22,KLO23} and Beltran-Guo-Hickman-Seeger \cite{BGHS}.
To achieve this goal, we employ a novel bootstrapping argument to establish a maximal estimate for the Mihlin-H\"{o}rmander-type multiplier, along with utilizing the local smoothing estimate for the averaging operator and its vector-valued extension to obtain crucial decay estimates. Furthermore, our approach offers an alternative means for deriving the upper bound established  in \cite{Guo20}.
\end{abstract}

\maketitle
\vskip .2in
\section{Introduction}
\label{s1}
There is an immense body of literature devoted to
 various important euclidean harmonic analysis problems associated with a surface or curve; for instance,
  restriction estimates \cite{G16,HKL20,HL14},
   $L^p$ estimates   of   averages and maximal averages \cite{G09,H16,S14,TW03,KLO22,KLO23,BGHS},   spherical averages  \cite{B86,S76},  Carleson maximal  operators \cite{L20,SW01,ZK21} and    singular integral operators \cite{GHLR,L19,LSY21,LY21,Wan19,Wan22}.
In this paper, we will consider a maximal function for families of Hilbert transforms along monomial curves in higher dimensions.

For an integer $d \geq 2$, let $\{\A_l\}_{l=1}^d$ be a sequence of distinct positive constants. Consider a monomial curve $\gamma: \mathbb{R} \to \mathbb{R}^d$ defined by $\gamma(s) := (s^{\A_1}, s^{\A_2}, \ldots, s^{\A_d})$.
\footnote{By convention, when $0<a\notin \Z$, the expression $s^a$ stands here for $|s|^a$ or for $sgn(s)|s|^a$ throughout this paper.}   The moment curve $(s, s^2, \ldots, s^d)$ is a typical example.
The Hilbert transform along the curve $u\gamma(s)$ $(u>0)$
acts on the Schwartz function $f$  by
\begin{equation}\label{hil}
H^{(u)}f(x)={\rm p.v.} \int_\R f(x+u\gamma(s))\frac{ds}{s}.
\end{equation}
For an arbitrary nonempty set $U\subset\R^+$, we will consider the following maximal function:
\begin{equation}\label{hil2}
\h^Uf(x)=\sup_{u\in U}|H^{(u)}f(x)|.
\end{equation}
This result that
 the individual operator $H^{(u)}$ is bounded on $L^p(\mathbb{R}^d)$ for $p \in (1, \infty)$, can be found in references such as \cite{SW78,DR86}. However, the maximal operator $\h^U$  is more intricate and requires further investigation.
The purpose of this paper is to obtain a sharp result for the $L^p$ operator norm of $\h^U$ defined by
$$\|\h^U\|_{L^p\to L^p}=\sup\{\|\h^U f\|_{L^p}:\ \|f\|_{L^p}\le 1\}$$
in relation to appropriate attributes of the set $U$.  To avoid cluttering the display, hereinafter we set
  \begin{equation}\label{alpha}
     1=\A_1<\A_2<\cdots<\A_d\hskip.2in  {\rm and} \hskip.2in\A_i\in\Z\hskip.1in{\rm for }\hskip.1in i=2,\cdots,d.
  \end{equation}

For the case of $d=2$, it can be shown  that $\h^U$ is equivalent to the maximal operator investigated by Guo-Roos-Seeger-Yung \cite{Guo20} through the change of  variable $s\to s/u$. They established a sharp bound $\sqrt{{\rm log}(e+\Re(U))}$ (up to a constant) for $p>2$,  where $\Re(U)$ is defined by
  \begin{equation}\label{bb1}
  \Re(U):= \#\{n\in\Z:\ [2^n,2^{n+1})\cap U\neq \emptyset\}.
  \end{equation}
   A bit more precisely, the upper bound in \cite{Guo20} was established by employing the local smoothing estimate for the wave-type operator and its square-function extension (see \cite{Se88}), as well as the maximal $L^p$ estimate for the Mihlin-H\"{o}rmander-type multiplier with respect to nonisotropic dilations. Additionally,  the lower bound in \cite{Guo20} was proven through the establishment of a crucial  generalization of Karagulyan's main theorem (see \cite{Ka07}).
 On the other hand, if we replace the curve $u\gamma(s)$ with the line $(s,us)$, the resulting operator $\h^U$ will be  the directional Hilbert transform denoted as $\h_L^U$, which is
 the primary focus of research in Stein's conjecture (see \cite{S87}).  Karagulyan \cite{Ka07} proved that there is a uniform constant $c>0$ such that  the $L^2\to L^{2,\infty}$ operator norm of the operator  $\h_L^U$ is bounded below by $c\sqrt{\log(\# U)}$.   \L aba, Marinelli and Pramanik \cite{LMP19} extended  this weak $L^2$ result   to all $L^p$ norms.
  Demeter \cite{D10},
  in particular,  proved  a sharp $L^2$  bound  ${\log(\# U)}$ (up to a constant).
   Afterwards, for the $L^p\to L^p$ operator norm of the operator $\h_L^U$, Demeter and Di Plinio \cite{DD14} found the upper bound  $C{\log(\# U)}$ for certain $C>0$ whenever $p>2$. Besides,  they obtained a sharp bound
  $\sqrt{\log(\#U)}$ (up to a constant) for   lacunary sets of direction as well as   some enhancements for Vargas-type direction sets.  Di Plinio and Parissis \cite{DP18} recently proved a similar result with regard to lacunary directions.
 We refer to \cite{ADP21,DP20,DGTZ18,KP22} and references therein for more significant developments on works related to the directional Hilbert transform.

However, there are few works on the $L^p$ boundedness of $\h^U$ in  higher dimensions, $d\ge3$, which is logically believed to be more challenging since  the related multipliers decay more slowly.
In reality, this problem is strongly connected to  the $L^p$ local smoothing estimate for the averaging operator  over the non-degenerate\footnote{The non-degenerate curve $\tilde{\g}$  means that $\tilde{\g}$ satisfies det$(\tilde{\g}',\tilde{\g}'',\cdots,\tilde{\g}^{(d)})(s)\neq 0$ on supp$\Psi_\circ$.} space curve $\tilde{\g}$  described by
$$ A_uf(x)=\int  f\big(x+u \tilde{\g}(s)\big)\Psi_\circ(s) ds,$$
where $\Psi_\circ$ is a bump function.    Moreover, by substituting  $s^{-1}$ with $\Psi_\circ(s)$ in (\ref{hil}),
one can basically link $\h^U$ to
 the maximal averaging operator  over the curve $\tilde{\g}$
$$\M f (x)=\sup_{u>0} |A_uf(x)|.$$
Indeed, the proof of  the desired estimate of   $\h^U$ will be dependent on  the $L^p$ estimate of $\M$
(or the $L^p$ local smoothing estimate of $A_u$). As a result, we mention some partial works on the $L^p$ estimate of $\M$ below.

The maximal averaging operators  over dilated submanifolds  have long been studied. According to
 Stein \cite{S76},   the spherical maximal function is  $L^p$ bounded if and only if $p>{d}/(d-1)$. Nearly a decade later, Bourgain \cite{B86} showed the  remainder $d=2$ using the approach of continuum incidence geometry; also see \cite{Sc97,Sc98,Lee03}.   The problem gets more complex for circles or curves with non-vanishing curvature in $\R^2$ since the conventional interpolation reasoning, which is heavily dependent on the $L^2$ estimate, no longer works.
  Pramanik and Seeger \cite{PS07} demonstrated for the first time that $\M$ is $L^p$ bounded for $p>p_w/2+1$ whenever the $l^p$ decoupling inequality (see \cite{W00,BD15}) holds for $p>p_w$ for $d=3$. Utilizing
   Bourgain and Demeter's
    $l^p$ decoupling inequality   on the optimal range $p\ge 6$ in \cite{BD15}, one can derive that $\M$ is $L^p$ bounded for $p>4$. Recently, Ko-Lee-Oh \cite{KLO22} and Beltran-Guo-Hickman-Seeger  \cite{BGHS}   demonstrated, respectively, that  $\M$ is $L^p$ bounded for the optimal range of $3<p\le \infty$
     based on two independent approaches: the $L^p$-$L^q$  smoothing estimate and the $L^p$ local smoothing estimate.
 For $d\ge 4$,  Ko-Lee-Oh \cite{KLO23}  established the $L^p$ boundedness of $\M$ for $p>2d-2$ by developing  the sharp local smoothing estimate in higher  dimensions.  This,   coupled with the fact that $\M$ cannot be bounded on $L^p$ if $p\le d$, as shown through a straightforward adaptation  of the reasoning in  \cite{KLO22}, suggests that the $L^p$ boundedness of $\M$   for $p\in(d,2d-2]$  remains unresolved  for $d\ge4$; see, for example, \cite{Guth}. In what follows,
  $p_\circ(d)$ denotes  the   best exponent known for   the $L^p$ boundedness of $\M$.
 Specifically,
  we can see from \cite{BGHS,KLO22,KLO23} that
 \begin{equation} \label{pdd}
p_\circ(d)=\left\{
\begin{array}{l}
d,\ \ \hskip.3in\ {\rm if}\ d=2,3,  \\
2d-2,\ \ \ {\rm if}\ d\ge 4.
\end{array}
\right.
\end{equation}
 Motivated by the work in \cite{Guo20} on  the two-dimensional  $\h^U$  and recent advances in \cite{BGHS,KLO22,KLO23} on the the averaging operator and the maximal averaging operator, we are interested in the following question:

\smallskip
{\bf Question:}\   Is the sharp bound for the $L^p\to L^p$ operator norm of $\h^U$
 in higher dimensions valid for any $p\in(p_\circ(d),\infty)$?

\smallskip
We now state the main result of this paper, which gives an affirmative answer to this question.
  \begin{thm}\label{t1}
  Let $d\ge3$.
  For each $p\in (p_\circ(d),\infty)$, the operator $\h^U$ defined by (\ref{hil2}) is $L^p$ bounded if and only if $\Re(U)<\infty$. Moreover, we have
  $$\|\h^U\|_{L^p\to L^p}\sim \sqrt{{\rm log}(e+\Re(U))}.$$
  \end{thm}
  \begin{rem}
We  list some comments on our main result.
\begin{itemize}
\item
Our technique works for $d=2$ as well,  we let $d\ge 3$ in Theorem \ref{t1} to stress the novel  part in the current study. Furthermore,
we  can verify  the lower bound for  all $p\in(1,\infty)$ (see Section \ref{sec8}), and the upper bound fails for $p\le d$ by modifying the counterexample  in \cite{KLO22}.

 \item
  The arguments stated here also hold
 for generic monomial  curves $\gamma(s)=([s]^{\A_1},\cdots,[s]^{\A_d})$, where $[s]^\A=c_+s^\A$ if $s>0$ and $[s]^\A=c_-(-s)^\A$ if $s<0$ for certain nonzero constants $c_\pm$. In fact, this form of curve was investigated in  \cite{Guo20}.

  \item
The subject of whether the monomial curves studied here can be generalized to more universal curves is interesting. It does, however, go beyond our techniques and prior methodologies in \cite{KLO22,KLO23,BGHS}.

\item
 Another intriguing topic is what occurs in the range $p\in (1, p_\circ(d))$. As far as we know, the two-dimensional case was addressed in \cite{GRSY20} with the premise that $U$ meets an extra sparseness condition. We expect that a similar sparseness requirement would be necessary for higher dimensions as well, but we opt not to study this topic because the dichotomy for the range of $p$ is unknown for $d\ge4$.

\end{itemize}
\end{rem}
\smallskip
\noindent{\bf Comments on the upper bound.} In the proof of the upper bound,  we will utilize
 two main novelties:
\begin{itemize}
    \item
    developing a bootstrapping  approach to show  a maximal estimate for the  Mihlin-H\"{o}rmander-type multiplier,
    which plays a crucial role in the proofs of Theorems \ref{t4.1} and \ref{1e};

    \item
    making full use of the  local smoothing estimate  for the averaging operator
    obtained by \cite{KLO22,KLO23,BGHS}
    and establishing its associated  vector-valued extension to achieve the desired decay in the proof of Theorem \ref{2e}.
\end{itemize}
By incorporating these strategies with an important inequality of the Chang-Wilson-Wolff kind, we can achieve  the desired estimate.
Additionally, since our method operates in two dimensions, it provides an alternative approach to addressing the associated difficulty discussed in \cite{Guo20}. 
Here we list  some comments on the proofs of Theorems \ref{t4.1} and \ref{2e}, as well as comparisons between our method and that of \cite{Guo20}.

\begin{itemize}
    \item

It is difficult to follow the procedure in \cite{Guo20} in reducing  the multiplier to the one without any dilations
 because the isotropic dilations have an effect on all variables at this point. Moreover,  this rationale, along with the Mihlin-H\"{o}rmander condition for anisotropic dilations, complicates the situation. To  address this difficulty, we employ a new bootstrapping argument that provides $d-1$ approximations of the original multiplier and eventually reduces the issue to a simple estimate; for more information, see Section \ref{s4}.

\item
In the proof of Theorem \ref{2e}, we deviate from the approach used in \cite{Guo20}, where the $L^p$ local smoothing estimate for the wave-type operator and its vector-valued extension were the main techniques employed to obtain the upper bound. Instead, we utilize the $L^p$ local smoothing estimate for the averaging operator and its vector-valued extension.
By employing these estimates, we are able to establish key square-function estimates. This is accomplished by introducing a cutoff function in (\ref{B}) and subsequently deriving a significant point-wise inequality (see (\ref{key}) below).
By leveraging this new approach, we are able to establish the desired result without relying on the wave-type operator's local smoothing estimate. 

\end{itemize}

\smallskip

\noindent{\bf Comments on the lower bound.}
  We first  construct   two desired approximations of the associated  multiplier with respect to  acceptable unbounded sets
   by  establishing   desired  decay estimates for certain oscillatory integrals (see Section \ref{sec8}), and then  utilize a Karagulyan-type theorem (see Proposition \ref{pp30}) from \cite{Guo20} to obtain the desired result.
   In particular, it is the choosing of these acceptable unbounded sets  that permits us to use only two approximations.

\medskip

\noindent{\bf Organization of the paper.} In Section \ref{s2},  we  reduce the proof of  the upper bound in Theorem \ref{t1}
to proving Theorems \ref{1e} and  \ref{2e}.
In Section \ref{sec.3}, we
provide  some auxiliary results, such as
the local smoothing estimates for the generic averaging operator and its vector-valued extension, the Chang-Wilson-Wolff-type inequality, some point-wise  inequalities for the martingale difference operator, and the H\"{o}rmander-type multiplier theorem.
Section \ref{s4} establishes  a crucial maximal  estimate that is used to prove Theorem \ref{1e}. Section \ref{s5} and Section \ref{s6}  give the proofs of Theorem \ref{2e} and Theorem  \ref{1e}, respectively.
In Section \ref{s7},
we show a maximum estimate of the Hilbert transform for lacunary sets.
In the last section, we demonstrate the lower bound in Theorem \ref{t1}.

\medskip

\noindent{\bf Notation.}   For any two quantities $x$ and $y$, we will write $x\lesssim y$ and $y\gtrsim x$ to denote   $x\le Cy$ for some absolute constant $C$. Subscripts will be used if the implied constant $C$ must be dependent on additional parameters. For example, $x\les_\rho y$ denotes
$x\les C_\rho y$ for some $C_\rho$ depending on $\rho$.
 If both $x\lesssim y$ and $x\gtrsim y$ hold, we use $x\sim y$. To abbreviate the notation, we will sometimes permit the implied constant to depend on certain fixed parameters (such as $\A_l$) when the issue of uniformity with respect to such parameters is irrelevant.
 The Fourier transform of a function $f$ is represented as $\mathcal{F}\{f\}$ or $\widehat{f}$, while the Fourier inverse transform of a function $g$ is represented by $\mathcal{F}^{-1}\{g\}$ or $\check{g}$.
 More precisely, we write
$$\mathcal{F}\{f\}(\xi)=\hat{f}(\xi)=\int_{\R^d} f(x)e^{-i\xi\cdot x} dx\ \ {\rm and}\ \    \mathcal{F}^{-1}\{g\}(x)
=\check{g}(x)=(2\pi)^{-d}\int_{\R^d} g(\xi)e^{-i\xi\cdot x} d\xi.$$
 Throughout  this paper, we omit the constant $(2\pi)^{-d}$ from  the  Fourier inverse transform  for  convenience.
 In some places of this paper,
 we   use  $|S|$ to  represent the Lebesgue measure of the set $S$, and
$\|\cdot\|_p$ to stand for $\|\cdot\|_{L^p(\R^d)}$.
Throughout this article, two cutoff functions $\psi: \mathbb{R} \to [0,1]$ and $\psi_\circ: \mathbb{R} \to [0,1]$ are fixed. The function $\psi$ has support on $\pm[1/2,2]$, while $\psi_\circ$ has support on $\pm[1/4,4]$. Additionally, $\psi_\circ$ is equal to one on the support of $\psi$.
\section{Reduction of the upper bound in  Theorem \ref{t1}}
\label{s2}
In this section, we reduce the proof of  the upper bound in  Theorem \ref{t1} to proving Theorems
\ref{1e} and  \ref{2e} below by employing an effective decomposition for $H^{(u)}$ (see (\ref{nde}) below).
Particularly, we  need to be extremely careful in selecting a smooth function  with compact support in the decomposition for $H^{(u)}$, which will play a crucial role in the proof of Theorem \ref{2e}.

\subsection{Littlewood-Paley decomposition}\label{sub2.1}
For $k\in\Z$, let
 $P_k$ be the usual  Littlewood-Paley projection on $\R^d$ with
$\widehat{P_kf}(\xi)=\psi(2^{-k}|\xi|)\widehat{f}(\xi)$,
 and write  $\F\{P_{\le k}f\}(\xi)=\phi(2^{-k}|\xi|)\widehat{f}(\xi)$, where the function $\psi$ is given as in {\bf Notation}, and the function $\phi$  satisfies
\begin{equation}\label{lpd}
 \phi(2^{-k}|\xi|)+\sum_{j>k}\psi(2^{-j}|\xi|)=1\ {\rm  for\  all}\  \xi\in\R^d.
 \end{equation}
 Obviously, we can write $f=P_{\le k}f+\sum_{j>k}P_kf$ for any $k\in\Z$ that we call the Littlewood-Paley decomposition of the function  $f$.
 Similarly, for $i=1,\cdots,d$ and $k\in\Z$,  we denote by $P_k^{(i)}$
  the  Littlewood-Paley projection in the $x
_i$-variable  on $\R$ with
$\F\{P_k^{(i)}f\}(\xi)=\psi(2^{-k}|\xi_i|)\widehat{f}(\xi),$ and   write
  $\F\{P_{\le k}^{(i)}f\}(\xi)=\phi(2^{-k}|\xi_i|)\widehat{f}(\xi)$. We thus also have
  $f=P_{\le k}^{(i)}f+\sum_{j>k}P_k^{(i)}f$.
  By employing a standard modification, we can extend the aforementioned definitions to encompass the case of $k\in\R$. More precisely,  if $k$ is not an integer,  we denote
$$P_k=P_{[k]},\ \ P_k^{(i)}=P_{[k]}^{(i)},\ \ P_{\le k}=P_{\le[k]},\ \ P_{\le k}^{(i)}=P_{\le[k]}^{(i)},$$
where $[\cdot]$ is the Gauss rounding function.
\subsection{Decomposition of the multiplier}\label{subss1}
Let $\{\delta_b^\A\}_{b\in\Z}$ be a dilation group defined by
\beq\label{v1}
\delta_b^{\A}(\xi)=(\frac{\xi_1}{2^{b\A_1}},
\frac{\xi_2}{2^{b\A_2}},\cdots,\frac{\xi_d}{2^{b\A_d}}),\ \ \hskip.3in  \xi=(\xi_1,\cdots,\xi_d)\in\R^d,
\eeq
where $\A=(\A_1,\cdots,\A_d)$ is given by (\ref{alpha}). Let
  $s^{-1}=\sum_{j\in\Z}\rho_j(s)$, where $\rho_j(s)=2^j\rho(2^js)$ with  $\rho(t)=\rho_0(t)$ is a smooth odd function supported in $\{s\in\R:\ 2^{-j-1}\le |s|\le 2^{-j+1}\}$.  Hence, we can write
$H^{(u)}f(x)=\sum_{j\in\Z}\int f(x+u\g(s))\rho_j(s)ds$.
Since  $\{\A_i\}_{i=1}^d\subset\Z$, \footnote{This assumption is to give a better presentation. Indeed, if $\A_i$ is not an integer for some $i$, we   use  $s^{-1}\chi_{s>0}=\sum_{j\in\Z}\rho_j^+(s)$ and $s^{-1}\chi_{s<0}=\sum_{j\in\Z}\rho_j^-(s)$ instead, where $\rho^+(s)$ is supported on $[1/2,2]$ and $\rho^-(s)=\rho^+(-s)$.}    using    the Fourier inverse transform and the change of  variable $s\to 2^{-j}s$, we have
$$H^{(u)}f(x)=\sum_{j\in\Z}\int_{\R^d} \widehat{f}(\xi)e^{i\xi\cdot x}m\left(u~\delta_j^\A(\xi)\right)d\xi, \ \ {\rm where}\ m(\xi):=\int e^{i\xi\cdot \g(s)} \rho(s)ds.$$
   Employing (\ref{lpd}) with $k=0$,  we can perform a further decomposition of $H^{(u)}f$ in the Fourier domain by splitting the symbol $m(\xi)$ into $A(\xi)$ and $B(\xi)$. More precisely,
\beq\label{m987}
m(\xi)=A(\xi)+B(\xi),
\eeq
where $A(\xi)$ and $B(\xi)$ are  given  by
\begin{align}
  A(\xi):=&\ m(\xi) \phi(|\xi|)+\sum_{l\ge 1}\psi(2^{-l}|\xi|)\int e^{i\xi\cdot \g(s)} \Upsilon^c\big(2^{-l} \g'(s)\cdot \xi\big)\rho(s)ds\  \label{A}\ {\rm and}\\
B(\xi):=&\sum_{l\ge 1} B_l(\xi),\ {\rm where}\  B_l(\xi)=\psi(2^{-l}|\xi|)\int e^{i\xi\cdot \g(s)}\Upsilon\big(2^{-l} \g'(s)\cdot \xi\big)\rho(s)ds.\   \label{B}
\end{align}
Here $\Upsilon$ in (\ref{B}) is a smooth function supported in $\{|\xi|\le c_0\}$ with sufficiently small  $c_0$ ($c_0=(9d)^{-1}2^{-2\A_d}$ is enough), and $\Upsilon^c=1-\Upsilon$.  We provide a heuristic explanation for the choice of $\Upsilon$. In fact, this particular choice can result in the existence of a pair $(l', l'') \in \{1,2,\cdots,d\}^2$ with $l' \neq l''$, such that the support of $B_l(\xi)$ is contained within a desired ``cube" in the $\xi_{l'}\xi_{l''}$ plane. Specifically,
if  $c_0$ is sufficiently small,  using
$1/2\le |s|\le2$, $ |\xi|\sim 2^l$ and $|\g'(s)\cdot \xi|
\le c_02^l$,  we can  deduce by a routine calculation  that
$|\xi_{l'}|\sim 2^l$ and $|\xi_{l''}|\sim 2^l$ hold simultaneously for some $l'\neq l''$.   The desired
$l'$ and $l''$, however, may depend on $\xi$. To fill this gap, we are  establishing  a crucial  point-wise estimate, see (\ref{yuel}) and Lemma \ref{l9} below. More importantly, this procedure will play an important role in closing the essential  square-function estimates in Section \ref{s6}.


We end this subsection by  explaining  the heuristic for the above decomposition (\ref{m987}). Indeed,  we can prove that $A(\xi)$ is a Schwartz function with $A(0)=0$. Precisely, it is clear  that
the first term  on the right-hand side of (\ref{A}) is   a Schwartz function which vanishes at the origin (since $m(0)=0$);  moreover,
observing the inequality  $|\xi\cdot \g'(s)|\gtrsim 2^l$ on supp$_\xi$ $\Upsilon^c\big(2^{-l} \g'(s)\cdot \xi\big)$, we can infer by     integrating  by parts that
$|\int e^{i\xi\cdot \g(s)} \Upsilon^c\big(2^{-l} \g'(s)\cdot \xi\big) \rho(s)ds\  \psi(2^{-l}|\xi|)|\les_N
2^{-Nl}$
for any $N\in\Z^+$, which yields  that
the second term on the right-hand side of (\ref{A}) is also a Schwartz function vanishing  at the origin. Regarding $B(\xi)$, we will attain the desired estimate by utilizing  the local smoothing estimate for the averaging operator, as well as its vector-valued extension.
\subsection{Reduction of the upper bound}\label{sub2.3}
 For every $u>0$, we  define two auxiliary operators  $S_u$ and $T_u:=\sum_{l\ge1}{T_u^{(l)}}$ by
 $$
\begin{aligned}
\F\{S_uf\}(\xi)=&\ \sum_{j\in\Z} A\left(u~\delta_j^\A(\xi)\right)\widehat{f}(\xi)\hskip.2in{\rm and}\hskip.2in
\F\{T_u^{(l)}f\}(\xi)=\ \sum_{j\in\Z} B_l\left(u~\delta_j^\A(\xi)\right)\widehat{f}(\xi).
\end{aligned}
$$
 This combined with  (\ref{m987}) implies that  for each $u>0$,
\begin{equation}\label{nde}
H^{(u)}f(x)=S_uf(x)+T_uf(x)=S_uf(x)+\sum_{l\ge 1}T_u^{(l)}f(x).
\end{equation}
To prove  the upper bound in Theorem \ref{t1},  it suffices to show the following theorems.
\begin{thm} \label{1e}
Let $d\ge 3$. For $p\in(1,\infty)$, then we have
\begin{equation}\label{ab21}
\big\|\sup_{u\in U}|S_uf|\big\|_p\les \sqrt{{\rm log}(e+\Re(U))} \big\|f\big\|_p,
\end{equation}
where $\Re(U)$ is given by (\ref{bb1}).
\end{thm}
\vskip.1in
\begin{thm} \label{2e}
Let $d\ge 3$ and $l\ge1$. For $p\in (p_\circ(d),\infty)$, there is a positive constant $\e_0$ such that
\begin{equation}\label{2ee}
\big\|\sup_{u>0}|T_u^{(l)}f|\big\|_p\les 2^{-l\e_0} \big\|f\big\|_p.
\end{equation}
\end{thm}
\section{Auxiliary results}
\label{sec.3}
In this section, we will introduce the local smoothing estimate  for a generic averaging operator, an inequality of the
Chang-Wilson-Wolff type and a point-wise inequality for the martingale difference operator.
\subsection{Local smoothing estimate  and its  vector-valued extension}\label{3.1}
 Let $\Gamma:[-2,2]\to \R^d$ be a smooth  curve satisfying the non-degenerate condition
\begin{equation}\label{nond}
{\rm det}(\Gamma'(s),\Gamma''(s),\cdots,\Gamma^{(d)}(s))\neq 0\ \ {\rm on}\ [-2,2].
\end{equation}
  For every $l\in\N$, we denote by  $a_l(s,t,\xi)$ a smooth function on
$[-2,2]\times [1/2,4]\times \{\xi\in\R^d:\ |\xi|\sim 2^l\}$, which
satisfies  the symbol condition
\begin{equation}\label{nond1}
|\p_s^j\p_t^k\p_\xi^\beta a_l(s,t,\xi)|\les |\xi|^{-|\beta|}
\end{equation}
for all $\beta\in \N^d$ and all $(k,j)\in\N^2$. Then we define an integral operator by
\begin{equation}\label{ll12}
A^\Gamma[a_l]f(x,t):= \int_{\R^d} {\bf m}_l(\xi,t)\widehat{f}(\xi)e^{i\xi\cdot x}d\xi,\ {\rm where}\ {\bf m}_l(\xi,t)=\int e^{-it\Gamma(s)\cdot \xi} a_l(s,t,\xi) ds.
\end{equation}
We next gives a point-wise estimate of the multiplier ${\bf m}_l(\xi,t)$, which plays an important role in proving the desired estimate in  $L^2$ norm. To be more specifically,
 applying the method of stationary phase, we can infer from  (\ref{nond}) and (\ref{nond1}) that
 \beq\label{zq1}
 |{\bf m}_l(\xi,t)
|\les (1+t|\xi|)^{-1/d}.
\eeq
It is evident that when $d$ grows, the decay rate of the multiplier ${\bf m}_l(\xi,t)$  drops. Furthermore, this observation might suggest that the corresponding task in higher dimensions is more complex.
\vskip.1in
The related local smoothing estimate for $A^\Gamma[a_l]$ is then given, which will be employed in the proof of Theorem \ref{2e}.
\begin{lemma}\label{l1}
Let $d\ge 3$,   $I=[1/2,4]$ and $l\ge 0$. For each $p\in(p_\circ(d),\infty)$,
\begin{equation}\label{cls}
\|A^\Gamma[a_l]f\|_{L^p(\R^d\times I)}\les 2^{-l(1/p+\e_1)}\|f\|_{L^p(\R^d)}
\end{equation}
holds for some $\e_1>0$.
\end{lemma}
\begin{proof}[Proof of Lemma \ref{l1}]
We first prove the $L^\infty$ estimate for $A^\Gamma[a]f$, that is,
 \begin{equation}\label{triv}
\|A^\Gamma[a_l]f\|_{L^\infty(\R^d\times I)}\les \|f\|_{L^\infty(\R^d)}.
 \end{equation}
Write $A^\Gamma[a_l]$ as a convolution operator below
 $$A^\Gamma[a_l]f(x,t)=\int_{-2}^2\int_{\R^d}K_{s,t}^l(x-y) f(y)dy ds,\ \ (x,t)\in \R^d\times I,$$
 where $K_{s,t}^l(x)=\int e^{i(x-t\Gamma(s))\cdot \xi} a_l(s,t,\xi)d\xi$.
 By
changing  the variable $\xi\to 2^l\xi$ and  integrating  by parts in $\xi$ not less than $d+1$ times,  we deduce from    (\ref{nond1}) that
 $$|K_{s,t}^l(x)|=
2^{ld}|\int e^{i2^l(x-t\Gamma(s))\cdot \xi} a_l(s,t,2^l\xi)d\xi|\les
\frac{2^{ld}}{(1+2^l|x-t\Gamma(s)|)^{d+1}}.$$
Then (\ref{triv}) follows from
$|\int_{\R^d}K_{s,t}^l(x-y) f(y)dy |\les \|f\|_\infty$ and $|I|\les1$.
\vskip.1in
For $d=3$, (\ref{cls}) for $3<p\le 4$ was proved in \cite{BGHS} (see Theorem 1.2 there).   Interpolating this with  (\ref{triv}), we  obtain (\ref{cls}) for $p\in(p_\circ(d),\infty)$  since $p_\circ(3)=3$.
For $d\ge 4$, Ko-Lee-Oh in \cite{KLO23}
 proved  that for $p\in (4d-2,\infty)$,
 \begin{equation}\label{ef1}
\|A^\Gamma[a_l]f\|_{L^p(\R^d\times I)}\les 2^{-l({2}/{p}-\e)}\|f\|_{L^p(\R^d)}
 \end{equation}
 holds for any $\e>0$.
Using (\ref{pdd}), we then see that (\ref{cls}) is a direct  consequence  by interpolating (\ref{ef1}) with
 \begin{equation}\label{ef2}
\|A^\Gamma[a_l]f\|_{L^2(\R^d\times I)}\les 2^{-{l}/{d}}\|f\|_{L^2(\R^d)}.
 \end{equation}
 Next, we prove  (\ref{ef2}).
By
 Plancherel's theorem,  we obtain
 \begin{equation}\label{ef22}
\|A^\Gamma[a_l]f\|_{L^2(\R^d\times I)}=(\int_{1/2}^4\|{\bf m}_l(\xi,t)
 \widehat{f}(\xi)\|_{L^2_\xi}^2dt)^{{1}/{2}},
 \end{equation}
which yields (\ref{ef2})  by inserting  (\ref{zq1}) into (\ref{ef22}). This completes the proof of Lemma \ref{l1}.
 \end{proof}
 Below we shall state  the square-function inequality with respect to anisotropic dilations, which  plays  a crucial role in the proof of Theorem \ref{t4.1}.  For each $j\in\Z$, we define the operator $T_{l,j}$ by
 $$\widehat{T_{l,j}f}(\xi,t)={\bf m}_l\big(2^l\delta_j^\A(\xi),t\big)\widehat{f}(\xi),\ l\ge 1,$$
 where ${\bf m}_l$ and $\delta_j^\A$ are   defined by (\ref{ll12}) and (\ref{v1}) for $b=j$, respectively.  By rescaling, Lemma \ref{l1} gives  that
 \beq\label{wan1}
 \|T_{l,j}f\|_{L^p(\R^d\times I)}\les  2^{-l({1}/{p}+\e_1)}\|f\|_{L^p(\R^d)},\ \ \ p\in(p_\circ(d),\infty).
 \eeq
 \begin{lemma}\label{l2}
 Let $d\ge 3$, $l\ge 1$ and $p\in(p_\circ(d),\infty)$.  Then there is an $\e>0$ such that
 \begin{equation}\label{smee}
 \|(\sum_{j\in\Z}|T_{l,j}f_j|^2)^{1/2}\|_{L^p(\R^d\times I)}\les  2^{-l({1}/{p}+\e)}\|(\sum_{j\in\Z}|f_j|^2)^{1/2}\|_{L^p(\R^d)}.
\end{equation}
 \end{lemma}
 \begin{proof}[Proof  of Lemma \ref{l2}]
  We denote  $\CC_j=\{\xi\in \R^d:\ 2^{-1}\le |\de_j^\A(\xi)|\le 2\}$, which satisfies that there is a positive  integer $n_\circ$  such that
 $\CC_j\cap \CC_{j'}=\emptyset$ whenever $|j-j'|\ge n_\circ$. Then, we  split the set of integers $\Z$ into $n_\circ$ sets,
  denoted as $\Lambda_0,\cdots,\Lambda_{n_\circ-1}$.  These subsets are defined as follows:
 $$\Lambda_k:=\{j\in\Z:\ j\equiv k\  ({\rm mod}\  n_\circ)\},\ \ k=0,1,\cdots,n_\circ-1.$$
 Consequently,  (\ref{smee})  follows from Minkowski's inequality  if
 \begin{equation}\label{go}
\|\big(\sum_{j\in \Lambda_k}|T_{l,j}f_j|^2\big)^{1/2}\|_{L^p(\R^d\times I)}
 \les 2^{-l(\e+\frac{1}{p})}\|(\sum_{j\in\Z}|f_j|^2)^{1/2}\|_{L^p(\R^d)}\ {\rm for\ all\ }   k=0,\cdots,{n_\circ-1}.
 \end{equation}
 We only prove (\ref{go}) for the case $k=0$, namely,
  \begin{equation}\label{sme}
\|\big(\sum_{j\in \Lambda_0}|T_{l,j}f_j|^2\big)^{1/2}\|_{L^p(\R^d\times I)}\les  2^{-l({1}/{p}+\e_0)}\|(\sum_{j\in\Z}|f_j|^2)^{1/2}\|_{L^p(\R^d)},
\end{equation}
 since $1\le k\le n_\circ-1$ can be treated analogously.
 Let $\{r_i(\tau)\}_{i=0}^\infty$ be the sequence of Rademacher functions (see, e.g.,  \cite{GG14}) on $[0,1]$,  which satisfy  that   for every $q\in(0,\infty)$,
 \begin{equation}\label{ab09}
 \|\sum_{i=0}^\infty z_i r_i(\tau)\|_{L^q_\tau([0,1])}\sim (\sum_{i=0}^\infty|z_i|^2)^{1/2},
 \end{equation}
and let 
  ${\bf m}_\tau^l(\xi,t):=\sum_{k=0}^\infty r_k(\tau){\bf m}_l(2^l\delta_{n_\circ k}^\A(\xi),t)$. By a routine computation,
we can obtain  from   (\ref{zq1}) and (\ref{wan1}) that for $p\in(p_\circ(d),\infty)$ and $s>0$,
$$
\begin{aligned}
|{\bf m}_\tau^l(\xi,t)|\les&\  2^{-l/d},\\
\|\F^{-1}\{\psi(|\cdot|) {\bf m}_\tau^l({\bar\delta}_s \cdot,t)\hat{f}\}\|_p\les&\  2^{-l(1/p+2\e)}\|f\|_p\ \ {\rm for\ some\ small}\ \e>0\ {\rm and}\\
|\p_\xi^\beta\big(\psi (|\xi|) {\bf m}_\tau^l({\bar\delta}_s \xi,t)\big)|\les&\  2^{l(d+1)}
\ \ \ \ \hskip.4in\ {\rm for\ any}\ |\beta|\le d+1,
\end{aligned}
$$
where ${\bar\delta}_s$ is defined by ${\bar\delta}_s:=\exp(s\log{\mathcal{P}})$  with  the  matrix $\mathcal{P}$ whose eigenvalues have positive real parts.
Applying Proposition 3.5 in \cite{Guo20} (see \cite{Se88} for the detailed proof) to the multiplier ${\bf m}_\tau^l(\xi,t) $
 and  the function $F$ defined by $\widehat{F}(\xi)=\sum_{k=0}^\infty \psi(\delta_{n_\circ k}^\A(\xi))\widehat{f_{n_\circ k}}(\xi)$ which satisfies
 $\|F\|_p\les \|(\sum_{k=0}^\infty|f_{n_\circ k}|^2)^{1/2}\|_p$,
     we can infer by taking the $L^p_\tau([0,1])$ norm on both sides of the resulting inequality that for $p\in(p_\circ(d),\infty)$,
 \begin{equation}\label{1q1}
 \begin{aligned}
 \big(\int_0^1 \int_{1/2}^2\|\F^{-1}\{{\bf m}_\tau^l(\cdot,t)\widehat{F}\}\|_{L^p(\R^d)}^p\  dt d\tau\big)^{1/p}
 \les&\  2^{-l(1/p+\e)}\|(\sum_{k=0}^\infty|f_{n_\circ k}|^2)^{1/2}\|_{L^p(\R^d)}.
 \end{aligned}
 \end{equation}
  Note that   $T_{l,n_\circ k}F=T_{l,n_\circ k}f_{n_\circ k}$ for all $k=0,1,\cdots$.  Using
  Fubini's theorem  and (\ref{ab09}) with $q=p$ to (\ref{1q1}), we have
   \begin{equation}\label{sme1}
\|\big(\sum_{k=0}^\infty |T_{l,n_\circ k}f_{n_\circ k}|^2\big)^{1/2}\|_{L^p(\R^d\times I)}\les  2^{-l(1/p+\e)}\|(\sum_{k=0}^\infty|f_{n_\circ k}|^2)^{1/2}\|_{L^p(\R^d)}.
\end{equation}
 Analogously,  following the above arguments, we may deduce
   \begin{equation}\label{sme2}
\|\big(\sum_{k=0}^\infty |T_{l,-n_\circ k}f_{-n_\circ k}|^2\big)^{1/2}\|_{L^p(\R^d\times I)}\les  2^{-l(1/p+\e)}\|(\sum_{k=0}^\infty|f_{-n_0k}|^2)^{1/2}\|_{L^p(\R^d)}.
\end{equation}
Note
$\Lambda_0=n_\circ \Z$. Finally,
 (\ref{sme}) follows by combining (\ref{sme1}) and (\ref{sme2}).
  \end{proof}
 \subsection{An inequality of the Chang-Wilson-Wolff type}\label{subCWW}
 For $j\in\Z$, we denote by $\Da_j^{(d)}$ the set of all dyadic cubes in $\R^d$ whose side length is $2^{-j}$.\footnote{The intervals are of the form $\prod_{i=1}^d[n_i2^{-j},(n_i+1)2^{-j})$ with $n=(n_1,\cdots,n_d)\in\Z^d$.} We define
 the conditional expectation of a locally integrable function $f$ on $\R^d$  by
$$E_jf(x)=\sum_{Q\in \Da_j^{(d)}} \frac{1}{|Q|}\int_Q f(y) dy \chi_Q(x),$$
and we write    the dyadic martingale difference operator $D_j$ and
 the dyadic square function $\mathfrak{D}f(x)$  by
$$D_jf(x)=E_{j+1}f(x)-E_jf(x) \hskip.2in {\rm and} \hskip.2in  \mathfrak{D}f(x)=(\sum_{j\in\Z}|D_jf(x)|^2)^{1/2}.$$
For $1\le l\le d$, $x\in\R^d$ and $y\in\R$, we let  ${\hat x}_l(y)$  denote the vector in $\R^d$  produced by
$x$ with the $l$-th component  $x_l$ replaced by $y$.
Analogously, we denote  $E_j^{(l)}$,  $D_j^{(l)}$ and  $\mathfrak{D}^{(l)}f$   by 
$$
\begin{aligned}
E_j^{(l)}f(x)=&\ \sum_{Q\in \Da_j^{(1)}} \frac{1}{|Q|}\int_Q f\left({\hat x}_l(y)\right) dy\chi_Q(x),\\
D_j^{(l)}f(x)=&\ E_{j+1}^{(l)}f(x)-E_j^{(l)}f(x)\ \ {\rm and}\ \  \mathfrak{D}^{(l)}f(x)=(\sum_{j\in\Z}|D_j^{(l)} f(x)|^2)^{1/2},
\end{aligned}$$
where
$\Da_j^{(1)}$ stands for the set of all dyadic intervals of length $2^{-j}$.
  Let $M$ denote the Hardy-Littlewood maximal operator, and let $M^{(l)}$ $(1\le l\le d)$ represent  the Hardy-Littlewood maximal operator in the $l$-th variable. For $\kappa\in \Z^+$, $1<q<\infty$ and $1\le l\le d$, we denote
$$M^\kappa:=\underbrace{M \circ M\circ\cdots\circ M}_{\kappa \ {\rm times}},\hskip.3in M_q(f)(x):=\big(M(|f|^q)\big)^{1/q}\hskip.1in {\rm and}\ \  M^{(l)}_q(f):= \big(M^{(l)}(|f|^q)\big)^{1/q}.$$
The following  forms of the Chang-Wilson-Wolff  inequality \cite{CWW85} play a crucial role in the proofs of Theorems \ref{1e} and \ref{t4.1}.
In particular, in Section \ref{s4}, we shall frequently  use the second inequality in  (\ref{a.2}).
 \begin{prop}\label{pp1}
Suppose that $f\in L^p(\R^d)\cap L^\infty(\R^d)$ for some $p<\infty$. Then there are two universal constants
$c_1$ and $c_2$ such that for all $\lambda>0$ and all $\e\in (0,1/2)$,
\begin{align}\label{a.2}
\begin{aligned}
|\{x\in\R^d:\ |f(x)|>4\lambda,\ \mathfrak{D}f(x)\le \e \lambda\}|
\le&\  c_2e^{-c_1\e^{-2}}|\{x\in\R^d:\ Mf(x)>\lambda\}|,\\
|\{x\in\R^d:\ |f(x)|>4\lambda,\ \mathfrak{D}^{(l)}f(x)\le \e \lambda\}|
\le&\ c_2e^{-c_1\e^{-2}}|\{x\in\R^d:\ M^{(l)}f(x)>\lambda\}|.
\end{aligned}
\end{align}
\end{prop}

Since this proposition can be shown following the proof of  Proposition 3.1 in \cite{Guo20},
 we omit  its proof.
 \subsection{Useful point-wise inequalities}
 \label{sub3.3}
The following lemmas provide  crucial point-wise estimates of $E_j$, $D_j$, $\mathfrak{D}$ and $\mathfrak{D}^{(l)}$.

\begin{lemma}\label{chu}
Let $d\ge1$, $j\in\Z$ and   $f\in L^1(\R^d)+L^\infty(\R^d)$. Then the following inequalities hold.\\
(1) For $q\ge1$ and $n\ge 0$,
$$E_j\big(\F^{-1}\{{\psi}(2^{-j-n}|\cdot|)\widehat{f}\}\big)(x)
\les 2^{-n(1-1/q)}M_q(Mf)(x).$$
(2) For $n\ge 0$, we have
$$D_j\big(\F^{-1}\{{\psi}(2^{-j+n}|\cdot|)\widehat{f}\}\big)(x)
\les 2^{-n}M^2f(x).$$
\end{lemma}
\begin{proof}
Let $\phi$ be the smooth function  given as in  Subsection \ref{sub2.1}, and
let $\phi_\circ(x)=\phi(16|x|)$. We define a function $\zeta$
by
$\check{\zeta}(x)=\phi_\circ(x)-2^d\phi_\circ(2x)$. So $\check{\zeta}$ is a smooth even function satisfying
supp $\check{\zeta}\subset\{x\in\R^d:|x|\le 1/8\}$, $\zeta(0)=0$ and  $|\zeta|\ge c$
in $\{\xi \in\R^d:1/8\le |\xi|\le 8\}$ for some $c>0$ (by the uncertainty principle).
Setting $\tilde{\psi}(\xi):=\frac{\psi_\circ(|\xi|)}{\zeta^2(\xi)}\in C_0^\infty(\R^d)$, we then have
 $\tilde{\psi}(\xi)\zeta^2(\xi)=1$ on  supp$\psi$, which clearly implies that  $\psi(|\xi|)\tilde{\psi}(\xi)\zeta^2(\xi)=\psi(|\xi|)$. Thus, there is a Schwartz function $W$ such that
\begin{equation}\label{am1}
{\psi}(2^{-j-n}|\xi|)\widehat{f}(\xi)=\zeta(2^{-j-n}\xi)\F\big\{W_{j+n}*f\big\}(\xi),
\end{equation}
where $ W_{j+n}(x)=2^{d(j+n)}W(2^{j+n}x)$ satisfies  $|W_{j+n}*f|(x)\les Mf(x)$.
Following the foregoing preparations, we deduce by using Sublemma 4.2 in   \cite{GHS06}  and (\ref{am1}) that
\begin{equation}\label{p09}
\begin{aligned}
E_j\big(\F^{-1}\{{\psi}(2^{-j-n}|\cdot|)\widehat{f}\}\big)(x)\les&\  2^{-n(1-1/q)}M_q( W_{j+n}*f)(x)\ {\rm and}\\
D_j\big(\F^{-1}\{{\psi}(2^{-j+n}|\cdot|)\widehat{f}\}\big)(x)
\les&\ 2^{-n} M(W_{j+n}*f)(x).
\end{aligned}
\end{equation}
At last, we may finish the proof of Lemma  \ref{chu} by using $|W_{j+n}*f|(x)\les Mf(x)$ to each inequality in  (\ref{p09}).
\end{proof}

We can deduce from Lemma \ref{chu} that for $q\in (1,\infty)$ and  $n\in\Z$,   there is an $\e\in (0,1-1/q)$ such that
\begin{equation}\label{am2}
\begin{aligned}
D_j\big(\F^{-1}\{{\psi}(2^{-j+n}|\cdot|)\widehat{f}\}\big)(x)
\les&\  2^{-|n|\e}M_q(Mf)(x)\ \ {\rm and}\ \\
D_j^{(l)}\big(\F^{-1}\{{\psi}(2^{-j+n}\xi_l)\widehat{f}(\xi)\}\big)(x)
\les&\  2^{-|n|\e}M_q^{(l)}(M^{(l)}f)(x),\ l=1,\cdots,d,
\end{aligned}
\end{equation}
which will be applied to  the following lemma. Clearly, the smooth function $\psi$ in (\ref{am2}) can be replaced by any smooth function supported in $\{\xi_l\in \R: |\xi_l|\sim1\}$.
\begin{lemma}\label{l45}
Let  $d\ge 1$, $l=1,\cdots,d$, $q\in (1,\infty)$,  and let $g\in L^1(\R^d)+L^\infty(\R^d)$. Then
\begin{equation}\label{ee2}
\mathfrak{D}^{(l)}g\les_q\Big(\sum_{j\in\Z}\big(M_q^{(l)}(M^{(l)}P_j^{(l)} g)\big)^{2}\Big)^{1/2}.
\end{equation}
\end{lemma}
\begin{proof}
We first define the operators
$\tilde{P}_{i}^{(l)}$ $(i\in\Z)$ by $\F\{\tilde{P}_{i}^{(l)} f\}(\xi)=\psi_\circ(2^{-i}\xi_l)\widehat{f}(\xi)$.
Note that
$D_j^{(l)}=\sum_{n\in\Z}D_j^{(l)} P_{j-n}^{(l)}\tilde{P}_{j-n}^{(l)}$.
 Using
 Minkowski's inequality and (\ref{am2}), we then deduce that
$$
\begin{aligned}
\mathfrak{D}^{(l)}g
\les&\  \sum_{n\in\Z}2^{-|n|\e}\Big(\sum_{j\in\Z}\big(M_q^{(l)}(M^{(l)}P_{j-n}^{(l)} g)\big)^{2}\Big)^{1/2}
\end{aligned}
$$
for certain $\e>0$, which yields (\ref{ee2})  immediately.
\end{proof}
\subsection{H\"{o}rmander-type multiplier theorem}
\label{sub3.4}
 We shall utilize   the following H\"{o}rmander-type multiplier theorem with respect to anisotropic dilations (see (\ref{n08})), which will be used in the proof of Proposition \ref{lac} below.
Let $\beta>0$, $i=1,2,\cdots,d$, and
define the operator   $J_{(i)}^\beta$  by $\F\{J_{(i)}^\beta f\}(\xi)=(1+|\xi_i|^2)^{\beta/2}\widehat{f}(\xi)$.
Remember that $\psi_\circ$ is given in {\bf Notation}.
\begin{prop}\label{mtmt}
    Let $\beta>1$, and let $m$ be a bounded function. There is a constant $C_H>0$ such that
\begin{equation}\label{n08}
    \sup_{t_1,t_2,\cdots,t_d>0}\|
   J_{(1)}^\beta\cdots J_{(d)}^\beta\Big(m(t_1\xi_1,\cdots,t_d\xi_d)
   \prod_{i=1}^d\psi_\circ(\xi_i)\Big) \|_1
   \le C_H.
\end{equation}
Then for every $p\in(1,\infty)$, the inequality  $\|\F^{-1}\{m(\xi)\hat{f}\}\|_p
\les C_H\|f\|_p$
holds with the constant  $C_H$  as in (\ref{n08}).
\end{prop}
\begin{rem}\label{qq09}
By Theorem 1.1 in \cite{GS19} (also see   \cite{CS92,CS95} for some related works), the assumption
(\ref{n08}) can be substituted  by
$$ \sup_{t_1,t_2,\cdots,t_d>0}\|
   J_{(1)}^\beta\cdots J_{(d)}^\beta\Big(m(t_1\xi_1,\cdots,t_d\xi_d)\prod_{i=1}^d
   \psi_\circ(\xi_i)\Big) \|_2
   \le C_H,$$
   with $\beta>1/2$. Here we do not seek the minimal assumptions on the number of derivatives because the condition (\ref{n08}) suffices  for the demonstration of Theorem  \ref{2e}.
\end{rem}
\begin{proof}
Using the partition of  unity  to  each variable $\xi_l$, we can write $m(\xi)$ as
$$m(\xi)=m(\xi)\sum_{{\bf j}\in\Z^d}\psi(\frac{\xi_i} {2^{j_1}})\cdots\psi(\frac{\xi_d}{ 2^{j_d}}),\ \ {\bf j}:=(j_1,\cdots,j_d). $$
Then, applying  the Littlewood-Paley theory, we deduce that for $p\in(1,\infty)$,
\begin{equation}\label{a76}
    \|\F^{-1}\{m(\xi)\hat{f}\}\|_p
    \les \|\big(\sum_{{\bf j}\in\Z^d}
    |K_{{\bf j}}*(P_{j_1}^{(1)}\cdots P_{j_d}^{(d)} f)|^2\big)^{1/2}\|_p,
\end{equation}
where
$K_{{\bf j}}(x):=\F^{-1}\big\{\psi_\circ(2^{-j_1}\xi_1)\cdots\psi_\circ(2^{-j_d}\xi_d)m(\xi)\big\}
(x).$
Note that the assumption (\ref{n08}) implies
$$|\F^{-1}\{\psi_\circ(\xi_1)\cdots\psi_\circ(\xi_d)m(2^{j_1}\xi_1,\cdots,2^{j_d}\xi_d)\}|(x)
\les C_H\ \prod_{i=1}^d\frac{1}{(1+|x_i|)^\beta},$$
which with the change of   variables $\xi_i\to 2^{-j_i}\xi_i$ $(1\le i\le d)$ leads to
$K_{{\bf j}}(x)
\les C_H\ \prod_{i=1}^d\frac{2^{j_i}}{(1+|2^{j_i}x_i|)^\beta}$. We thus have by a routine calculation that for all ${\bf j}\in\Z^d$,
\begin{equation}\label{91}
    |K_{{\bf j}}*g|(x)\les C_H\ M^{(1)}\cdots M^{(d)}g(x).
\end{equation}
Finally, we may infer the required conclusion from  Khintchine's inequality and the Marcinkiewicz multiplier theorem by entering  (\ref{91}) into (\ref{a76}).
\end{proof}
\section{A crucial maximal  estimate on $\R^d$}
\label{s4}
In this section, we present a maximal  estimate for the Mihlin-H\"{o}rmander-type multiplier which is
one of
the  novelties in this paper. More precisely,
the challenge is in employing the condition (see  (\ref{cod1}) below) for the  Mihlin-H\"{o}rmander-type multiplier with respect to the anisotropic scaling to control the  maximal operator (see (\ref{i00}) and (\ref{w1}) below) in terms of the isotropic scaling in all variables.
As in the previous statements, we shall develop a bootstrapping argument to overcome this difficulty.
\vskip.1in
   For $s\in \R$ and $q\in [1,\infty]$, we define  the $W^{s,q}(\R^d)$ (Sobolev space) norm of a function $f$ by $\|J^sf\|_q$,
where the operator $J^s$ is defined by $\widehat{J^sf}(\xi):=(1+|\xi|^2)^{s/2}\widehat{f}(\xi)$.
Let  $\{\A_i\}$ be defined by (\ref{alpha}), and let $\delta_j^\A$ be the dilation  defined   by (\ref{v1}) with $b=j$.
To relate  the anisotropic scaling with the  isotropic  scaling, we  introduce a new distance
$$\|\xi\|:=\big((\xi_1^2)^\frac{\A_d}{\A_1}+(\xi_2^2)^\frac{\A_d}{\A_2}+\cdots+(\xi_d^2)^\frac{\A_d}{\A_d}\big)^\frac{1}{2\A_d}.$$
Note that    $\|\delta_j^\A(\xi)\|=2^{-j}\|\xi\|$,  $\|\xi\|\sim \sum_{l=1}^d|\xi_l|^\frac{1}{\A_l}$ and  $\psi(\|\xi\|)\in C_0^\infty(\R^d)$.
\begin{thm}\label{t4.1}
 Let $d\ge 1$ and $n\in\Z$.
Suppose that $a(\xi)$ satisfies $a(0)=0$ and
\begin{equation}\label{cod1}
\sup_{j\in\Z}\Big\|\psi(\|\xi\|)a\big(\delta_{j}^\A(\xi)\big)\Big\|_{W^{d+3,1}(\R^d)}\les 1.
\end{equation}
Let $\NN$ be a subset of $\Z$ with $\#\NN=N_0$, and let  $\T_n$  denote an operator defined by
\beq\label{i00}
\widehat{\T_nf}(\xi):=a(2^n\xi)\widehat{f}(\xi).
\eeq
Then  for each  $p\in (1,\infty)$,
\begin{equation}\label{w1}
\|\sup_{n\in\NN}  |\T_n f(x)|\|_p\les \sqrt{\log(e+N_0)}\|f\|_p,
\end{equation}
holds with  the implicit constant  independent of $\NN$.
\end{thm}
\begin{rem}
The Sobolev space  in (\ref{cod1}) can be slightly improved to be   $W^{d+2+\e_0,1}(\R^d)$ with $\e_0\in(0,1)$, but  we choose not to pursue  this direction because this operation has no effect on the proof.
\end{rem}
Before we give the proof, we need some new notations.  For each  $l\in [2, d]\cap \Z$,  we write
 $\xi_{1,l}:=(\xi_1,\cdots,\xi_l)\in\R^l.$ Then we have $\xi_{1,d}=\xi=(\xi_1,\cdots,\xi_d)$.
For convenience, we write  $(\xi_{1,d},0)=\xi$, and $(\xi_{1,l},0)=(\xi_{1,l},0,\cdots,0)\in\R^d$ whenever $1\le l\le d-1$. Besides, for $q\in (1,\infty)$, if $1\le i<j\le d$, we write
\begin{equation}\label{maxi}
M^{(i,j)}:=M^{(i)}\circ M^{(i+1)}\circ\cdots \circ M^{(j)}\ \ {\rm and}\ \  M^{(i,j)}_qf(x)=\big(M^{(i,j)}(|f|^q)(x)\big)^{1/q}.
\end{equation}
 For $i> j$ and $i=j$, we denote   $M^{(i,j)}:=Id$ (identity)    and  $M^{(i,j)}:=M^{(i)}$, respectively.
Define
\begin{equation}\label{isd}
a_j(\xi):=a(\delta_{-j}^\A(\xi))\psi(\|\xi\|),
\end{equation}
which satisfies
$|\F^{-1}\{a_j\}*f|\les Mf$ (since (\ref{cod1})). The proof of Theorem \ref{t4.1} is based  on the  following Proposition \ref{pq1} with respect to  the multipliers $\{m_{l,j,n}\}$ satisfying that for each $l\in [2, d]\cap \Z$ and every $(j,n)\in\Z^2$,
there is a positive constant $C_0$ independent of
$j,n$ such that
\beq\label{ads1}
|\F^{-1}\{m_{l,j,n}\widehat{f}\}|(x)\le C_0 M^{(l+1,d)}f(x).
\eeq
More importantly, Proposition \ref{pq1} will
 provide the conditions to the bootstrapping argument below.
\vskip.1in
\begin{prop}\label{pq1}
Let  $l\in [2, d]\cap \Z$, $(j,n,k)\in\Z^3$,  $p\in (1,\infty)$, and let  $\{m_{l,j,n}\}$ be the multipliers satisfying (\ref{ads1}). Define the sets $\{\Lambda_l^n(k)\}$ by $\Lambda_l^n(k)=\{j\in\Z:\ j\A_l>n+k\}$, and
define the operators $\{S_{l,k,n}\}$ by
$$\F\{S_{l,k,n}f\}(\xi)=\sum_{j\in \Lambda_l^n(k)}
\Big(a_j(2^n\delta_j^\A(\xi_{1,l},0))-a_j(2^n\delta_j^\A(\xi_{1,l-1},0))\Big)m_{l,j,n}(\xi)\psi(2^{-k}\xi_l)
\widehat{f}(\xi)$$
with  $a_j$  given by (\ref{isd}).
Then
there are two positive  constants  $C_1$, independent of $k,n$, and $C_2$, independent of $n$, such that
\begin{align}
|S_{l,k,n}f(x)|\le&\  C_1 M^{(1,d)}(P_k^{(l)}f)(x),\label{ev1}\\
\|\sum_{k\in\Z}S_{l,k,n}f\|_p\le&\  C_2 \|f\|_p.\label{ev2}
\end{align}
\end{prop}
\begin{proof}[Proof of Proposition \ref{pq1}]
We first deduce (\ref{ev2}) from (\ref{ev1}). Applying
(\ref{ev1}), the Fefferman-Stein inequality and the Littlewood-Paley theory in order,
we can control  the left-hand side of (\ref{ev2}) by a  constant (uniformly in $k,n$) multiplied by
$$
\begin{aligned}
\|(\sum_{k\in\Z}|S_{l,k,n}f|^2)^{1/2}\|_p\les&\  \|(\sum_{k\in\Z}|M^{(1,d)}(P_k^{(l)}f)|^2)^{1/2}\|_p\\
\les&\   \|(\sum_{k\in\Z}|P_k^{(l)}f|^2)^{1/2}\|_p\les \|f\|_p.
\end{aligned}
$$
As a result, (\ref{ev2}) follows.
We next show  (\ref{ev1}).
Let $K_{j,k,n}$ be the function on $\R^l$ defined by
\begin{equation}\label{Ed09}
\begin{aligned}
\F\{K_{j,k,n}\}(\xi_{1,l},0)=&\ \Big(a_j\big(2^n\delta_j^\A(\xi_{1,l},0)\big)-a_j\big(2^n\delta_j^\A(\xi_{1,l-1},0)\big)\Big)\psi(2^{-k}\xi_{l}).
\end{aligned}
\end{equation}
Applying the fundamental theorem of calculus to (\ref{Ed09}), we infer by changing  the variable
$(\xi_{1,l},0)\to 2^{-n}\delta_{-j}^\A(\xi_{1,l},0)$ that
$$\widehat{K_{j,k,n}}(2^{-n}\delta_{-j}^\A(\xi_{1,l},0))= \psi(\frac{\xi_l}{ 2^{k+n-j\A_l}})\xi_l
\int_0^1 (\p_l a_j)(\xi_{1,l-1},s\xi_l,0)ds.
$$
Changing the variable $\xi_l\to 2^{n+k-j\A_l}\xi_l$,  we obtain from the inequality   $\|a_j\|_{W^{d+3,1}(\R^d)}\les1$ (since (\ref{cod1}) and (\ref{isd})) that
\begin{equation}\label{wss1}
\|\F\{K_{j,k,n}\}(2^{j\A_1-n}\xi_1,\cdots,2^{j\A_{l-1}-n}\xi_{l-1},2^k\xi_l,0) \|_{W^{l+1,1}(\R^l)}\les 2^{n+k-j\A_l},
\end{equation}
which implies
\begin{equation}\label{m2}
    |K_{j,k,n}(x_{1,l})|\les 2^{n+k-j\A_l}G_{l,j,k,n}(x_{1,l}),
    \ \ x_{1,l}=(x_1,\cdots,x_l),
\end{equation}
where    the function $G_{l,j,k,n}(x_{1,l})$ is  given  by
\begin{equation}\label{m1}
    G_{l,j,k,n}(x_{1,l}):=\frac{2^k}{(1+2^{k}|x_l|)^{1+1/l}}
    \prod_{i=1}^{l-1}\frac{2^{j\A_i-n}}{(1+2^{j\A_i-n}|x_i|)^{1+1/l}}.
\end{equation}
Note that the sum over $j\in\Lambda_l^n(k)$ can be  absorbed by
the factor $2^{n+k-j\A_l}$.
 Combining (\ref{m2}) and (\ref{m1}), we  then deduce from the application of the Fourier inverse transform that
$
|S_{l,k,n}f|(x)
\les M^{(1,d)}f(x)
$
with the implicit constant independent of $k,n$.
Finally, we can achieve (\ref{ev1})  by using $S_{l,k,n}f=S_{l,k,n}\tilde{P}_k^{(l)}f$ and applying  the above arguments to the operator $S_{l,k,n}\tilde{P}_k^{(l)}$ with $\tilde{P}_k^{(l)}$  as in the proof of Lemma \ref{l45}.
\end{proof}
\begin{rem}\label{rr21}
Clearly,
the above set $\Lambda_l^n(k)$    can be replaced by $\Lambda_l^n(k\pm C)$ with $0< C\les1$.
\end{rem}
\begin{proof}[Proof of Theorem \ref{t4.1}]
By interpolation,  it suffices to show that for each $p\in(1,\infty)$,
\begin{equation}\label{w2w0}
|\{x\in\R^d:\ \sup_{n\in\NN}  |\T_n f(x)|>4\lambda\}|\les \big( \lambda^{-1} \sqrt{\log(e+N_0)}\big)^p
\end{equation}
holds for all Schwartz functions $f$ with $\|f\|_p=1$, and all $\lambda>0$.
Since $a(0)=0$, we may decompose $\T_n f$ as
$$\widehat{\T_n f}(\xi)
=\sum_{k\in\Z}\sum_{j\in \Z}a_j(2^n\delta_j^\A(\xi))\psi(2^{-k}\xi_d)
\widehat{f}(\xi).$$
In fact, using the supports of $a_j(2^n\delta_j^\A(\xi))$ and $\psi(2^{-k}\xi_d)$, we deduce that there is a unform $C>0$ such that
$j\in \Lambda_d^n(k-C)$.
Let
$ V_{1,n}^{(d)}$ and $ V_{2,n}^{(d)}$ be  two operators defined by
$$
\begin{aligned}
    \F\{V_{1,n}^{(d)}f\}(\xi)=&\ \sum_{k\in\Z}\sum_{j\in \Lambda_d(k-C)}
\Big(a_j(2^n\delta_j^\A(\xi))-a_j(2^n\delta_j^\A(\xi_{1,d-1},0))\Big)\psi(2^{-k}\xi_d)
\widehat{f}(\xi)\ {\rm and}\\
\F\{V_{2,n}^{(d)}f\}(\xi)=&\ \sum_{k\in\Z}\sum_{j\in \Lambda_d(k-C)}
a_j(2^n\delta_j^\A(\xi_{1,d-1},0))\psi(2^{-k}\xi_d)
\widehat{f}(\xi),\ {\rm respectively.}
\end{aligned}
$$
Then
$
    \T_n f(x)=V_{1,n}^{(d)}f(x)+V_{2,n}^{(d)}f(x),
$
and we will obtain  the desired estimate  (\ref{w2w0})  if
  \begin{align}
|\{x\in\R^d:\ \sup_{n\in\NN}  |V_{1,n}^{(d)}f(x)|>4\lambda\}|\les&\  \big( \lambda^{-1} \sqrt{\log(e+N_0)}\big)^p\ {\rm and}\label{w2w2}\\
|\{x\in\R^d:\ \sup_{n\in\NN}  |V_{2,n}^{(d)}f(x)|>4\lambda\}|\les&\  \big( \lambda^{-1} \sqrt{\log(e+N_0)}\big)^p\label{w2w3}
\end{align}
hold for all Schwartz functions $f$ with $\|f\|_p=1$, and all $\lambda>0$.
To finish the proof of Theorem \ref{t4.1}, it remains to show (\ref{w2w2}) and (\ref{w2w3}). For the proof of (\ref{w2w2}),
we shall use Proposition \ref{pq1}, the inequality (\ref{a.2}) of the  Chang-Wilson-Wolff type,  and Lemma \ref{l45}.  For (\ref{w2w3}), however, its proof is more complicate since we  need more techniques to handle the sum of $j$. More precisely,
we require  some further careful  decompositions which will form a bootstrapping approach. Besides, the method  in \cite{Guo20}, which is based on a  Cotlar type inequality, does not work since the isotropic dilations have an effect on all variables at this point.

We begin with the estimate of (\ref{w2w2}).
Let $\e_*$  be a positive constant to be chosen later.
The left-hand side of (\ref{w2w2}) is
bounded by
$$
\begin{aligned}
&\sum_{n\in\NN} |\{x\in\R^d:\ |V_{1,n}^{(d)}f(x)|>4\lambda,\ \mathfrak{D}^{(d)}V_{1,n}^{(d)}f(x)\le \e_* \lambda\}|\\
+&\ |\{x\in\R^d:\ \sup_{n\in\NN}|\mathfrak{D}^{(d)}V_{1,n}^{(d)}f(x)|> \e_* \lambda\}|
=:\ \ I_1+I_2,
\end{aligned}
$$
where $\mathfrak{D}^{(d)}f$ is   the dyadic square function with respect to the martingale operator $D_j^{(d)}$ (see Subsection \ref{subCWW}).
Applying  (\ref{a.2}) with $l=d$ and Proposition \ref{pq1} with $m_{l,j,n}(\xi)\equiv 1$, we deduce
\begin{equation}\label{i1}
    \begin{aligned}
I_1\le&\  c_2N_0\  e^{-c_1\e_*^{-2}}\max_{n\in\NN} |\{x\in\R^d:\ M^{(d)}V_{1,n}^{(d)}f(x)>\lambda\}|\\
\les&\ N_0 \ e^{-c_1\e_*^{-2}}\lambda^{-p}\max_{n\in\NN}  \|M^{(d)}V_{1,n}^{(d)}f\|_p^p\\
\les&\ N_0\  e^{-c_1\e_*^{-2}}\lambda^{-p}
\end{aligned}
\end{equation}
with the constants $c_1,c_2$  as in Proposition \ref{pp1}.
Using Chebyshev's inequality, (\ref{ee2}) with $l=d$ and (\ref{ev1}), we then have
$$
\begin{aligned}
I_2\le&\  \e_*^{-p}\lambda^{-p}\| \sup_{n\in \NN} |\mathfrak{D}^{(d)}(V_{1,n}^{(d)}f)|\|_p^p\\
\les&\ \e_*^{-p}\lambda^{-p} \big\| \sup_{n\in \NN}\| M_q^{(d)}(M^{(d)}V_{1,n}^{(d)}P_k^{(d)}f) \|_{l^2_k} \big\|_p^p\\
\les&\ \e_*^{-p}\lambda^{-p} \big\| \| M_q^{(d)}(M^{(d)}M^{(1,d)}P_k^{(d)}f) \|_{l^2_k} \big\|_p^p
\end{aligned}
$$
where  $q$ will be chosen so that $ 1<q<\min\{p,2\}$.
Since $p\in(1,\infty)$, it follows from the Fefferman-Stein inequality and the Littlewood-Paley theory that
$\big\| \| M_q^{(d)}(M^{(d)}M^{(1,d)}P_k^{(d)}f) \|_{l^2_k} \big\|_p\les \|f\|_p,$
which yields
\begin{equation}\label{i2}
I_2\les \e_*^{-p}\lambda^{-p}.
\end{equation}
Combining (\ref{i1}) and (\ref{i2}),  we finally obtain  (\ref{w2w2}) by
setting   $\e_*^{-2}=c_1^{-1}\log N_0$.
\vskip.1in
It remains to show  (\ref{w2w3}). Since we do not have a small factor
 to absorb the sum over $j\in\Lambda_l^n(k-C)$, (\ref{w2w3}) requires a more intricate analysis.
Write $\phi_k:=\sum_{j\le k}\psi(2^{-j}\cdot)$.
By changing the order of the sums of $j$ and $k$,
we rewrite $\F\{V_{2,n}^{(d)}f\}(\xi)$ as
$$\F\{V_{2,n}^{(d)}f\}(\xi)=\sum_{j\in\Z}
a_j(2^n\delta_j^\A(\xi_{1,d-1},0))m_{d-1,j,n}(\xi)
\widehat{f}(\xi)$$
for some $m_{d-1,j,n}(\xi)$ satisfying
\begin{equation}\label{md1}
|\F^{-1}\{m_{d-1,j,n}(\xi)
\widehat{f}\}|(x)\les M^{(d)}f(x).
\end{equation}
Then, it follows by
applying the partition of unity $\sum_{k\in\Z}\psi(2^{-k}\xi_{d-1})=1$ that
$$
\begin{aligned}
    \F\{V_{2,n}^{(d)}f\}(\xi)
=&\ \sum_{k\in\Z}\sum_{j\in \Lambda_{d-1}^n(k-C)}a_j(2^n\delta_j^\A(\xi_{1,d-1},0))\psi(\frac{\xi_{d-1}}{2^{k}})m_{d-1,j,n}(\xi)
\widehat{f}(\xi)
\end{aligned}
$$
for some $C$.
Next, we further decompose $V_{2,n}^{(d)}$ by preforming  an  analogous  process  as splitting  $\T_n$ into $V_{1,n}^{(d)}$ and  $V_{2,n}^{(d)}$. Precisely, we have
$$V_{2,n}^{(d)}f(x)=V_{1,n}^{(d-1)}f(x)+V_{2,n}^{(d-1)}f(x),$$
where
 the operators
$ V_{1,n}^{(d-1)}$ and $ V_{2,n}^{(d-1)}$ are defined by
$$
\begin{aligned}
   \F\{V_{1,n}^{(d-1)} f\}(\xi):=&\ \sum_{k\in\Z}\sum_{j\in \Lambda_{d-1}^n(k-C)}
\Big\{a_j\big(2^n\delta_j^\A(\xi_{1,d-1},0)\big)-a_j\big(2^n\delta_j^\A(\xi_{1,d-2},0)
\big)\Big\}
\psi(\frac{\xi_{d-1}}{2^{k}})\\
&\ \ \ \ \ \ \ \ \ \ \ \times
m_{d-1,j,n}(\xi)
\widehat{f}(\xi)\ \ \ \ {\rm and}\\
\F\{V_{2,n}^{(d-1)}f\}(\xi):=&\ \sum_{k\in\Z}\sum_{j\in \Lambda_{d-1}^n(k-C)}
a_j(2^n\delta_j^\A(\xi_{1,d-2},0))\psi(\frac{\xi_{d-1}}{2^{k}})m_{d-1,j,n}(\xi)
\widehat{f}(\xi).
\end{aligned}
$$
Hence, (\ref{w2w3}) will follow if
\begin{align}
|\{x\in\R^d:\ \sup_{n\in\NN}  |V_{1,n}^{(d-1)}f(x)|>4\lambda\}|\les&\  \big( \lambda^{-1} \sqrt{\log(e+N_0)}\big)^p\ {\rm and}\ \label{cx1}\\
|\{x\in\R^d:\ \sup_{n\in\NN}  |V_{2,n}^{(d-1)}f(x)|>4\lambda\}|\les&\  \big( \lambda^{-1} \sqrt{\log(e+N_0)}\big)^p\label{cx2}
\end{align}
hold for all Schwartz functions $f$ with $\|f\|_p=1$, and all $\lambda>0$. Notice that
 we can get  (\ref{cx1}) by
arguing similarly as in the proof of  the  estimate  of $V_{1,n}^{(d)}$. A bit more precisely,  its proof is based on
   (\ref{a.2}) with $l=d-1$,   Proposition \ref{pq1} with $l=d-1$,
   and (\ref{md1}).  As a result, it remains to prove (\ref{cx2}). By a similar argument as the treatment of  $V_{2,n}^{(d)}$, we can write
$$\F\{V_{2,n}^{(d-1)}f\}(\xi)=\ \sum_{j\in\Z}
a_j(2^n\delta_j^\A(\xi_{1,d-2},0))m_{d-2,j,n}(\xi)
\widehat{f}(\xi)
$$
for some  $m_{d-2,j,n}(\xi)$ satisfying
$
|\F^{-1}\{m_{d-2,j,n}(\xi)
\widehat{f}\}|(x)\les M^{(d-1,d)}f(x).
$
Without loss of generality, we assume $d\ge4$ in what follows since otherwise we complete the proof by the simple estimate (\ref{endf}) below.
 For each $2\le l\le d-2$, we define
$ V_{1,n}^{(l)}$ and $ V_{2,n}^{(l)}$ by
$$
\begin{aligned}
   \F\{V_{1,n}^{(l)} f\}(\xi)=&\ \sum_{k\in\Z}\sum_{j\in \Lambda_{l}^n(k-C)}
\Big\{a_j\Big(2^n\delta_j^\A(\xi_{1,l},0)\Big)-a_j\Big(2^n\delta_j^\A(\xi_{1,l-1},0)\Big)\Big\}
\psi(\frac{\xi_{l}}{2^{k}})\\
&\ \ \ \ \ \ \hskip.2in\times m_{l,j,n}(\xi)
\widehat{f}(\xi)\  \ {\rm and}\\
\F\{V_{2,n}^{(l)}f\}(\xi)=&\ \sum_{k\in\Z}\sum_{j\in \Lambda_{l}^n(k-C)}
a_j(2^n\delta_j^\A(\xi_{1,l-1},0))\psi(\frac{\xi_{l}}{2^{k}})m_{l,j,n}(\xi)
\widehat{f}(\xi)
\end{aligned}
$$
for some $m_{l,j,n}(\xi)$ satisfying $|\F^{-1}\{m_{l,j,n}(\xi)
\widehat{f}\}|(x)\les M^{(l+1,d)}f(x)$.
   Repeating the above procedure $d-3$ times ($V_{2,n}^{(l)}=V_{1,n}^{(l-1)}+V_{2,n}^{(l-1)}$ for $l=d-1,d-2,\cdots,3$) with
   $$|\{x\in\R^d:\ \sup_{n\in\NN}  |V_{1,n}^{(l)}f(x)|>4\lambda\}|\les \big( \lambda^{-1} \sqrt{\log(e+N_0)}\big)^p,\ l=d-2,\cdots,2,$$
  we can reduce   the proof of Theorem \ref{t4.1} to showing  that for each $p\in(1,\infty)$,
   \begin{equation}\label{end09}
|\{x\in\R^d:\ \sup_{n\in\NN}  |V_{2,n}^{(2)}f(x)|>4\lambda\}|\les \big( \lambda^{-1} \sqrt{\log(e+N_0)}\big)^p
\end{equation}
holds for all Schwartz functions $f$ with $\|f\|_p=1$, and all $\lambda>0$, where the operator $V_{2,n}^{(2)}$ is given by
$$
\F\{V_{2,n}^{(2)}f\}(\xi)=\ \sum_{k\in\Z}\sum_{j\in \Lambda_{1}^n(k-C)}
a_j\Big(2^n\delta_j^\A(\xi_1,0)\Big)\psi(\frac{\xi_{1}}{2^{k}})m_{1,j,n}(\xi)
\widehat{f}(\xi)
$$
for some  $m_{1,j,n}(\xi)$ satisfying
$
|\F^{-1}\{m_{1,j,n}(\xi)
\widehat{f}\}|(x)\les M^{(2,d)}f(x).$ Since $a_j(0)=0$, we can rewrite
$V_{2,n}^{(2)}$ as
\beq\label{endf}
\F\{V_{2,n}^{(2)}f\}(\xi)=\ \sum_{k\in\Z}\sum_{j\in \Lambda_{1}^n(k-C)}
\Big\{a_j(2^n\delta_j^\A(\xi_1,0))-a_j(2^n\delta_j^\A(0))\Big\}
\psi(\frac{\xi_{1}}{2^{k}})m_{1,j,n}(\xi)
\widehat{f}(\xi),
\eeq
which obeys  a desired estimate by following the proof of   (\ref{w2w2}). This completes the proof of Theorem \ref{t4.1}.
\end{proof}
\begin{rem}\label{2dok}
The method in the proof of Theorem \ref{t4.1} which  works for all $d\ge1$ is quite different from  \cite{Guo20} whose idea is to reduce the proof of the desired estimate to proving  the uniform estimate for a lower-dimensional operator. Moreover,  with this theorem in hand, one may give an alternative proof to  the main result in \cite{Guo20}.
\end{rem}
\section{Proof of Theorem \ref{1e}}
\label{s5}
In this section, we shall show Theorem \ref{1e} using
the crucial Theorem \ref{t4.1}.
\begin{proof}[Proof of Theorem \ref{1e}]
For  $l\in\Z$, we denote  $\Phi_l(\xi):=\psi(2^{-l}|\xi|)A(\xi)$, and  define two operators  $S_u^{(l)}$ and $\Y_u^{(l)}$ by
$$
\begin{aligned}
    \F\{S_u^{(l)} f\}(\xi):=&\ \sum_{j\in\Z}\Phi_l(u\delta_j^\A(\xi))\widehat{f}(\xi)
    \ \ \ {\rm and}\ \ \
\F\{\Y_u^{(l)} f\}(\xi):=\ \sum_{j\in\Z}\Phi_l(2^lu\delta_j^\A(\xi))\widehat{f}(\xi).
\end{aligned}$$
Then we can decompose $S_u$ as   $S_u=\sum_{l\in\Z} S_u^{(l)}$. Thus, it is sufficient for   (\ref{ab21})  to prove that for  $p\in(1,\infty)$,
$$\|\sup_{u\in U}|S_u^{(l)}f|\|_p\les 2^{-|l|}\sqrt{{\rm log}(e+\Re(U))}  \|f\|_p.$$
Using the isotropic rescaling, we can achieve this from
\begin{equation}\label{aa211}
\|\sup_{u\in U}|\Y_u^{(l)}f|\|_p\les 2^{-|l|}\sqrt{{\rm log}(e+\Re(U))}  \|f\|_p,\ \ p\in(1,\infty).
\end{equation}
It thus remains to establish (\ref{aa211}). Let $\NN$ be the set defined by
 $\NN:=\{n\in\Z:\ [2^n,2^{n+1})\cap U\neq \emptyset\}$. Then   $\#\NN=\Re(U)$. Then we further reduce the matter to proving 
\begin{equation}\label{kill1}
\|\sup_{n\in \NN}\sup_{\tau\in[1,2)}|\Y_{2^n\tau}^{(l)}f|\|_p\les 2^{-|l|}\sqrt{{\rm log}(e+\# \NN)}  \|f\|_p,\ \ p\in(1,\infty).
\end{equation}
By  the fundamental theorem of calculus, to achieve (\ref{kill1}), it  suffices to prove
\begin{align}
\|\sup_{n\in \NN}|\Y_{2^n}^{(l)} f| \|_p\les&\  2^{-|l|}\sqrt{\log(e+\# \NN)} \|f\|_p\ \label{5.1} {\rm and}\\
\int_1^2 \|\sup_{n\in \NN}|\p_\tau(\Y_{2^n\tau}^{(l)} f)| \|_pd\tau \les&\  2^{-|l|}\sqrt{\log(e+\# \NN)} \|f\|_p.\label{5.2}
\end{align}
We  next show (\ref{5.1}) and (\ref{5.2}) in order. A routine computation gives that
$$\sup_{k\in\Z}\Big\|\psi(\|\xi\|)\sum_{j\in\Z}\Phi_l
\big(2^l\delta_{j+k}^\A(\xi)\big)\Big\|_{W^{s,1}(\R^d)}\les 2^{-|l|}$$
holds for all $s>0$. Then we deduce (\ref{5.1})  by
  applying Theorem \ref{t4.1} to the operator $2^{|l|}\Y_{2^n}^{(l)}$.  In addition, we get by
 a simple computation that
  $$\nabla \Phi_l=2^{-l}\psi'(2^{-l}|\xi|)\frac{\xi}{|\xi|}A(\xi)+\psi(2^{-l}|\xi|)(\na A)(\xi)=:\vec{\Phi}_{1,l}(\xi)+\vec{\Phi}_{2,l}(\xi).$$
From this equality we infer that
$$
\begin{aligned}
\F\{\p_\tau(\Y_{2^n\tau}^{(l)} f)\}(\xi)=&\ \tau^{-1}\sum_{i=1,2}\sum_{j\in\Z} 2^{l+n}\tau\delta_j^\A(\xi)\cdot  \vec{\Phi}_{i,l}(2^{l+n}\tau\delta_j^\A(\xi))\widehat{f}(\xi)\\
=&\
\tau^{-1}\sum_{i=1,2}a_i(2^n\xi)\widehat{f}(\xi),
\end{aligned}
$$
 with $a_i(\xi)=\sum_{j\in\Z} 2^{l}\tau\Big(\delta_j^\A(\xi)\cdot  \vec{\Phi}_{i,l}\big(2^{l}\tau\delta_j^\A(\xi)\big)\Big)$.
Using  $\tau\in[1,2]$ and the estimate
$$\sup_{k\in\Z}\|\psi(\|\xi\|)
a_i(\delta_k^\A(\xi))\|_{W^{s,1}(\R^d)}\les 2^{-|l|}$$
for all $s>0$, we infer by
 applying Theorem \ref{t4.1} to  $2^{|l|}\p_\tau(\Y_{2^n\tau}^{(l)}f)$ that for every
  $p\in (1,\infty)$,
$$
 \|\sup_{n\in \NN}|\p_\tau(\Y_{2^n\tau}^{(l)} f)| \|_p \les 2^{-|l|}\sqrt{\log(e+\# \NN)} \|f\|_p,
$$
which
 yields
(\ref{5.2}) immediately.
\end{proof}
\section{Proof of Theorem \ref{2e}}
\label{s6}
In this section, we will prove Theorem \ref{2e} by using Lemmas \ref{l1} and \ref{l2}. Before we go ahead, we need first  a lemma giving an essential  point-wise estimate which is  used to establish the inequalities (\ref{loc1}) and (\ref{loc2}) below.  Let $\{\A_l\},\g(s)$  be given as in Theorem \ref{t1},  $M^{(1,d)}$ be defined as in (\ref{maxi}) with $(i,j)=(1,d)$, and let $\de_j^\A$ be given as in (\ref{v1}) with $b=j$.
\begin{lemma}\label{l9}
Let  $c_0$ be a fixed positive  constant smaller than $(9d)^{-1}2^{-2\A_d}$,  and let $(j,k)\in\Z^2$. Suppose that  $h$ is a Schwartz function whose Fourier transform is supported in
\begin{equation}\label{sppo1}
\CC_{j,k}:=\bigcup_{s:|s|\in[1/2,2]}\{\xi\in\R^d:\ 2^{-k-3}\le |\de_j^\A(\xi)|\le 2^{-k+2},\ \  |\g'(s)\cdot \de_j^\A(\xi)|\le 2^{-k}c_0\}.
\end{equation}
Then  the following point-wise inequality
\begin{equation}\label{key}
|h|(x)\les \sum_{l'=1}^d\sum_{l''\in \{1,2,\cdots,d\}\setminus\{l'\}}|\PP_{j\A_{l'}-k}^{(l')}
\bar{\PP}_{j\A_{l''}-k}^{(l'')}(K_{l',l''}^{j,k}*h)|(x)
\end{equation}
holds for some kernel functions $\{K_{l',l''}^{j,k}\}$ with $|K_{l',l''}^{j,k}*h|(x)\les M^{(1,d)}\circ M^{(1,d)}h(x)$,
and for some operators  $\PP_{k}^{(n)}$ and $\bar{\PP}_{k}^{(n)}$, which are variants  of the Littlewood-Paley operator $P_{k}^{(n)}$.
\end{lemma}
\begin{proof}[Proof of Lemma \ref{l9}]
Let  $\Psi:\R\to [0,1]$ be a smooth even function supported in $\{z\in\R:\ (9d)^{-1}\le |z|\le {9d}\}$, which equals 1 in $\{z\in\R:\ (8d)^{-1}\le |z|\le {8d}\}$,
and let $\Psi^c=1-\Psi$.
By the first restriction  on the right-hand side of  (\ref{sppo1}), we have
$\hat{h}(\xi)\prod_{i=1}^d\Psi^c(2^{k-j\A_i}\xi_i)=0,$
which yields
\begin{equation}\label{mu}
\hat{h}(\xi)=\sum_{i=1}^d\hat{h}(\xi)\Psi(2^{k-j\A_i}\xi_i)m_{1,j,i,k}(\xi)
\end{equation}
for some $m_{1,j,i,k}(\xi)$ satisfying
\begin{equation}\label{dd09}
|\F^{-1}\{m_{1,j,i,k}(\xi)\hat{f}\}|(x)
\les M^{(1,d)}f(x).
\end{equation}
 (\ref{mu}) gives a preliminary decomposition of $\hat{h}$, however, not  desired.

Next,
we give a further decomposition of $\hat{h}(\xi)$  by splitting
each   $\hat{h}(\xi)
\Psi(2^{k-j\A_i}\xi_i)$.
Set $E=(2\A_d+2d)^{10(d+\A_d)}$ (this choice is enough but not optimal).
Let $\tilde{\Psi} $ denote  a  non-negative  smooth even function supported in $\{z\in\R:(2E)^{-1}\le |z|\le {2E}\}$, which equals 1 in $\{z\in\R:E^{-1}\le |z|\le {E}\}$,
and let $\tilde{\Psi}^c=1-\tilde{\Psi}$. For each $i\in \{1,2,\cdots,d\}$,
we  claim  by both restrictions on the right-hand side of  (\ref{sppo1}) that
\begin{equation}\label{kill2}
\hat{h}(\xi)
\Psi(2^{k-j\A_i}\xi_i)\prod_{n\in\{1,2,\cdots,d\}\setminus\{i\}}\tilde{\Psi}^c(2^{k-j\A_n}\xi_n)=0.    
\end{equation}
Indeed,  
   the first 
restriction in (\ref{sppo1}) yields   $|2^{k-j\A_n}\xi_n|\le4$ for each $n\in\{1,\cdots,d\}$,   while
the support of $\tilde{\Psi}^c(2^{k-j\A_n}\xi_n)$ $(n\neq i)$    leads to   $|2^{k-j\A_n}\xi_n|>E$ or $|2^{k-j\A_n}\xi_n|<E^{-1}$. So  
$|2^{k-j\A_n}\xi_n|< E^{-1}$
for $n\neq i$. On the other hand, since the support of $\Psi(2^{k-j\A_i}\xi_i)$ gives  $|2^{k-j\A_i}\xi_i|>(9d)^{-1}$,  we have   
$$|\g'(s)\cdot \de_j^\A(\xi)|\ge \A_i2^{1-\A_i}2^{-j\A_i}|\xi_i|-d\A_d2^{\A_d}E^{-1}2^{-k}
\ge \A_i2^{-\A_i}(9d)^{-1} 2^{-k}
> 2c_02^{-k},$$ 
which contradicts with the second restriction in (\ref{sppo1}),  hence the support of the left hand side of (\ref{kill2}) is $\emptyset$, and (\ref{kill2}) holds.  
Now, we obtain from (\ref{kill2})  that
\begin{equation}\label{dec1}
\begin{aligned}
   \hat{h}(\xi)
\Psi(2^{k-j\A_i}\xi_i)
=&\ \hat{h}(\xi)
\Psi(2^{k-j\A_i}\xi_i)\Big(1-\prod_{n\in\{1,2,\cdots,d\}\setminus\{i\}}\tilde{\Psi}^c(2^{k-j\A_n}\xi_n) \Big)\\
=&\ \hat{h}(\xi)
\Psi(2^{k-j\A_i}\xi_i)\sum_{n\in \{1,2,\cdots,d\}\setminus\{i\}}
\tilde{\Psi}(2^{k-j\A_n}\xi_n)m_{2,j,n,k}(\xi)
\end{aligned}
\end{equation}
holds for some $m_{2,j,n,k}(\xi)$ satisfying
\begin{equation}\label{dd10}
|\F^{-1}\{m_{2,j,n,k}(\xi)\hat{f}\}|(x)
\les M^{(1,d)}f(x).
\end{equation}
Plugging (\ref{dec1}) into (\ref{mu}), we then have
\begin{equation}\label{yuel}
\hat{h}(\xi)=\hat{h}(\xi)\sum_{i=1}^d\sum_{n\in \{1,2,\cdots,d\}\setminus\{i\}}\Psi(2^{k-j\A_i}\xi_i)
\tilde{\Psi}(2^{k-j\A_n}\xi_n)m_{2,j,n,k}(\xi)m_{1,j,i,k}(\xi).
\end{equation}
Due to  (\ref{dd09}) and (\ref{dd10}) we  infer that the product of $m_{2,j,n,k}(\xi)$ and $m_{1,j,i,k}(\xi)$ satisfies
\begin{equation}\label{roue}
|\F^{-1}\{m_{2,j,n,k}(\xi)m_{1,j,i,k}(\xi)\hat{f}\}|(x)\les M^{(1,d)}\circ M^{(1,d)} f(x).
\end{equation}
In fact,  if we  expanded the implicit multipliers  $m_{2,j,n,k}(\xi)$ and $m_{1,j,n,k}(\xi)$, the right-hand side of (\ref{roue}) might be replaced by $M^{(1,d)} f$.  Nevertheless,
(\ref{roue}) is enough to achieve our goal.
For $ 1\le i,n\le d$, we define the operators
$\PP_k$ and $\bar{\PP}_k$ by
\beq\label{jiade}
\F\{\PP_k^{(i)}f\}(\xi):=\Psi(2^{-k}\xi_i)\hat{f}(\xi)\ {\rm and}\
\F\{\bar{\PP}_k^{(n)}f\}(\xi):=\tilde{\Psi}(2^{-k}\xi_n)\hat{f}(\xi).
\eeq
Writing  $K_{i,n}^{j,k}(x):=\F^{-1}\{m_{2,j,n,k}(\xi)m_{1,j,i,k}(\xi)\}(x)$ which is desired since (\ref{roue}), we
then obtain  (\ref{key})  by taking the Fourier inverse transform on both sides of (\ref{yuel}).
\end{proof}
\begin{proof}[Proof of Theorem \ref{2e}]
We begin with showing a square-function estimate.
For $j\in\Z$, $t\in [1/2,4]$ and  $l\ge 1$, we define
\begin{equation}\label{m77}
 \widehat{\Xi_{j,l}^t f}(\xi):=B_l(t\de_j^\A(\xi))\widehat{f}(\xi),
 \end{equation}
 where $B_l$ is given by (\ref{B}).
 Applying Lemma \ref{l1}  with  $\Gamma=\gamma$ and $a_l(s,t,\xi)=\Upsilon(\frac{t\g'(s)\cdot \xi}{2^l}) \rho(s)\psi(2^{-l}t|\xi|)$, we deduce that for $p\in (p_\circ(d),\infty)$, there is an $\e>0$ such that
 \begin{equation*}
 \|\Xi_{0,l}^{t} f\|_{L^p(\R^d\times I)}
\les 2^{-l({1}/{p}+\e)}\|f\|_{L^p(\R^d)},
 \end{equation*}
 which implies by  Lemma \ref{l2} with ${\bf m}_l(\xi,t)=B_l(t\xi)$ that for  $p\in (p_\circ(d),\infty)$,
 \begin{equation}\label{4.2}
 \|(\sum_j |\Xi_{j,l}^{2^lt} f_j|^2)^{{1}/{2}}\|_{L^p(\R^d\times I)}\les 2^{-l({1}/{p}+\e)}\|(\sum_j|f_j|^2)^{{1}/{2}}\|_{L^p(\R^d)}.
 \end{equation}
 By the rescaling,  we may obtain from (\ref{4.2}) that
  \begin{equation}\label{4.21}
 \|(\sum_j |\Xi_{j,l}^{2^kt} f_j|^2)^{{1}/{2}}\|_{L^p(\R^d\times I)}\les 2^{-l({1}/{p}+\e)}\|(\sum_j|f_j|^2)^{{1}/{2}}\|_{L^p(\R^d)}
  \end{equation}
  holds for any $k\in\Z$.
Note  that  the inequality (\ref{2ee}) equals
\begin{equation}\label{end321}
\|\sup_{n\in\Z}\sup_{v\in[1,2)}|T_{2^{n}v}^{(l)}f|\|_p\les 2^{-l\e_0}\|f\|_p.
\end{equation}
Let $\chi:\R\to [0,1]$ be a smooth function supported on $I$, which equals 1 on $[1,2]$. By
interpolation inequality
$$\sup_{v\in [1,2)}|g(v)|\le \|\chi(v)g(v)\|_{L^\infty_v}\les\ \|\chi(v)g(v)\|_{L^p_v}^{1-1/p}\|\frac{d}{dv}\big(\chi(v)g(v)\big)\|_{L^p_v}^{1/p}$$
 and $l^q\subset l^\infty$ for any $q\in [1,\infty)$,   it is sufficient for (\ref{end321}) to  prove that for each $p\in(p_\circ(d),\infty)$,
\begin{align}
(\sum_{n\in\Z}\|T_{2^{n}v}^{(l)}f\|_{L^p(\R^d\times I)}^p)^{{1}/{p}}\les&\   2^{-l({1}/{p}+\e)}\|f\|_{L^p(\R^d)}\ {\rm and} \label{loc1}\\
(\sum_{n\in\Z}\|\p_v(T_{2^{n}v}^{(l)}f)\|_{L^p(\R^d\times I)}^p)^{{1}/{p}}\les&\   2^{l-l({1}/{p}+\e)}\|f\|_{L^p(\R^d)}.
\label{loc2}
\end{align}
We next prove (\ref{loc1}) and (\ref{loc2}) in order.
\subsection{Proof of (\ref{loc1})}\label{sub6.1}
Rewrite $T_{2^{n}v}^{(l)}f$ as
$
T_{2^{n}v}^{(l)}f=\sum_{j\in\Z}\Xi_{j,l}^{2^nv}f.
$
Since
$$
\F\{\Xi_{j,l}^{2^nv}f\}(\xi)
=\hat{f}(\xi)\psi(2^{n-l}v|\de_j^\A(\xi)|)\int e^{i2^nv\de_j^\A(\xi)\cdot \g(s)}\Upsilon\big(2^{n-l}v \g'(s)\cdot \de_j^\A(\xi)\big)\rho(s)ds
$$
and $v\in I$,
the support of $\F\{\Xi_{j,l}^{2^nv}f\}(\xi)$ is a subset of
\begin{equation}\label{spport}
\bigcup_{s:|s|\in[1/2,2]}\left\{\xi\in\R^d:\ 2^{l-n-3}\le |\de_j^\A(\xi)|\le 2^{l-n+2},\ \  |\g'(s)\cdot \de_j^\A(\xi)|\le 2^{l-n+2}c_0\right\}
\end{equation}
with $c_0$ small enough (since the choice of $\Upsilon$ in Subsection \ref{subss1}).
As the previous analysis below (\ref{B}), we obtain via Lemma \ref{l9} with $k=n-l$ that
\begin{equation}\label{ll1}
|T_{2^{n}v}^{(l)}f|(x)\les \sum_{l'=1}^d\sum_{l''\in \{1,2,\cdots,d\}\setminus\{l'\}}\left|\sum_{j\in\Z}\Xi_{j,l}^{2^nv}
\PP_{j\A_{l'}+l-n}^{(l')}\bar{\PP}_{j\A_{l''}+l-n}^{(l'')}(K_{l',l''}^{j,n-l}*f)\right|(x)
\end{equation}
where
\begin{equation}\label{sxw1}
|K_{l',l''}^{j,n-l}*f|(x)\les M^{(1,d)}\circ M^{(1,d)}f(x).
\end{equation}
By (\ref{ll1}), we can deduce (\ref{loc1}) from
  \begin{equation}\label{loca1}
\Big(\sum_{n\in\Z}\|\sum_{j\in\Z}\Xi_{j,l}^{2^nv}\PP_{j\A_{l'}+l-n}^{(l')}
\bar{\PP}_{j\A_{l''}+l-n}^{(l'')}(K_{l',l''}^{j,n-l}*f)
\|_{L^p(\R^d\times I)}^p\Big)^{{1}/{p}}\les\  2^{-l({1}/{p}+\e)}\|f\|_{L^p(\R^d)}
  \end{equation}
 with $l'\neq l''$. In other words, to finish  the proof of (\ref{loc1}), it remains to show  (\ref{loca1}).
Applying the Littlewood-Paley theory, (\ref{sxw1}) and the Fefferman-Stein inequality, we have
\begin{equation}\label{6.7}
\begin{aligned}
&\ \|\sum_{j\in\Z}\Xi_{j,l}^{2^nv}\PP_{j\A_{l'}+l-n}^{(l')}
\bar{\PP}_{j\A_{l''}+l-n}^{(l'')}(K_{l',l''}^{j,n-l}*f)
\|_{L^p(\R^d\times I)}\\
\les&\  \|\Big(\sum_{j\in\Z}|\Xi_{j,l}^{2^nv}\PP_{j\A_{l'}+l-n}^{(l')}
\bar{\PP}_{j\A_{l''}+l-n}^{(l'')}f|^2\Big)^{1/2}\|_{L^p(\R^d\times I)}.
\end{aligned}
\end{equation}
By using  (\ref{4.21}) with $k=n$, we can bound  the right-hand side of  (\ref{6.7}) by a uniform constant multiplied by
$2^{-l({1}/{p}+\e)} \|\Big(\sum_{j\in\Z}|\PP_{j\A_{l'}+l-n}^{(l')}\bar{\PP}_{j\A_{l''}+l-n}^{(l'')
}f|^2\Big)^{{1}/{2}}\|_p.$
A similar argument gives that  the left-hand side of (\ref{loca1})
 is
 $$
 \begin{aligned}
\les&\ 2^{-l({1}/{p}+\e)}\Big(\sum_{n\in\Z} \|(\sum_{j\in\Z}|\PP_{j\A_{l'}+l-n}^{(l')}\bar{\PP}_{j\A_{l''}+l-n}^{(l'')}f
|^2\Big)^\frac{1}{2}\|_p^p\big)^{{1}/{p}}\\
\les&\  2^{-l({1}/{p}+\e)} \|\Big(\sum_{(j,n)\in\Z^2}
|\PP_{j\A_{l'}+l-n}^{(l')}\bar{\PP}_{j\A_{l''}+l-n}^{(l'')}f|^2\Big)^{{1}/{2}}\|_p,
\end{aligned}
$$
where we used  Fubini's theorem, $p\ge 2$ and $l^2\subset l^q$ for any $q\ge 2$. At last, since $\A_{l'}\neq \A_{l''}$, the desired (\ref{loca1}) follows from
  the Littlewood-Paley inequality.
\subsection{Proof of (\ref{loc2})}
\label{sub6.2}
By a routine computation, we can write $\na B_l(\xi)$ as
$$\na B_l(\xi)={\bf B}_{l,1}(\xi)+2^{-l}{\bf B}_{l,2}(\xi),$$
where the vector-valued functions  ${\bf B}_{l,1}(\xi)$ and ${\bf B}_{l,2}(\xi)$ are given by
$$
\begin{aligned}
{\bf B}_{l,1}(\xi):=&\ \psi(\frac{|\xi|}{2^{l}})\int e^{i\xi\cdot \g(s)}\Upsilon(\frac{\g'(s)\cdot \xi}{2^l}) \g(s)\rho(s)ds\\
{\bf B}_{l,2}(\xi):=&\ \psi(\frac{|\xi|}{2^{l}})\int e^{i\xi\cdot \g(s)}\Upsilon'(\frac{\g'(s)\cdot \xi}{2^l}) \g'(s)\rho(s)ds\\
&\ +\frac{\xi}{|\xi|}\psi'(\frac{|\xi|}{2^{l}})\int e^{i\xi\cdot \g(s)}\Upsilon(\frac{\g'(s)\cdot \xi}{2^l}) \rho(s)ds.
\end{aligned}$$
Then we have
\beq\label{po1}
\begin{aligned}
\p_v\Big(B_l\big(2^nv\de_j^\A(\xi)\big)\Big)=&\ 2^n\de_j^\A(\xi)\cdot (\na B_l)\big(2^nv\de_j^\A(\xi)\big)\\
=&\ 2^n\de_j^\A(\xi)\cdot {\bf B}_{l,1}\big(2^nv\de_j^\A(\xi)\big)
+v^{-1}2^{-l}2^{n}v\de_j^\A(\xi)\cdot {\bf B}_{l,2}\big(2^nv\de_j^\A(\xi)\big).
\end{aligned}
\eeq
Note that the first term on the right-hand side of (\ref{po1}) devotes to the
 the main contribution  since the second term is  similar but better by a factor $2^{-l}$. Therefore,
arguing similarly as in the proof of   (\ref{loc1}), we can also obtain
 the desired estimate of  $2^{-l}\p_v(T_{2^{n}v}^{(l)}f)$. This completes the proof of
 (\ref{loc2}).
 \end{proof}
\section{Maximal functions for lacunary sets}\label{s7}
In this section, we consider the operator norm of $\h^U$ for the lacunary sets, and establish $\|\h^U\|_{L^p\to L^p}$ for some $p\in (1,2]$, which will be used to show the lower bound of Theorem \ref{t1} in Section \ref{sec8}.
\begin{define}
Let $\la>1$. A finite set $U$ is called $\la$-lacunary if it can be arranged in a sequence $U=\{u_1<u_2<\cdots<u_M\}$ where $u_j\le u_{j+1}/\la$ for $1\le j\le M-1$.
\end{define}
For $\la>1$, we denote $C_\la:=\max\{1,\log_{\la}2\}$. Let $ \Re(U)$ be defined by (\ref{bb1}). Then
 $\Re(U)
\le \# U \le  C_\la \Re(U)$ whenever $U$ is $\la$-lacunary.
\begin{prop}\label{lac}
Let $U$ be a $\la$-lacunary set. Then for each $p\in (\frac{2d^2}{d^2+1},\infty)$,
\begin{equation}\label{adx1}
\|\h^U\|_{L^p\to L^p}\les C_\la\sqrt{\log(e+\# U)}.
\end{equation}
\end{prop}
\begin{rem}\label{rr100}
Although the range of $p$ can be slightly improved to be $p\in(\frac{2d^2}{d^2+2},\infty)$ by
using  the condition in $Remark$ \ref{qq09} (see Subsection \ref{sub3.4}),  the range of $p$ in Proposition \ref{lac} is enough  in the following proof.
   \end{rem}
\begin{proof}
Since $U$ is $\la$-lacunary, we can split $U$ into $N(\les C_\lambda)$ sets which are denoted by $\{U_i\}_{i=1}^N$
such that there is at most one element in $U_i\cap [2^n,2^{n+1})$ for each pair $(i,n)\in \{1,\cdots,N\}\times \Z$. Without loss of generality, we  just consider the lower bound of $\|\h^{U_1}\|_{L^p\to L^p}$ since   $\{U_i\}_{i=2}^N$ can be treated  similarly.

Now, we order $U_1=\{u_k\}$ such that $u_k<u_{k+1}$,  and denote by $n(k)$ the unique integer $n$ such that $u_k\in I_n$. To obtain (\ref{adx1}), it suffices to prove that for
$p\in (\frac{2d^2}{d^2+1},\infty)$,
\begin{equation}\label{adx2}
\|\h^{U_1}\|_{L^p\to L^p}\les \sqrt{\log(e+\# U_1)}.
\end{equation}
Write  $H^{(u)}=S_u+\sum_{l\ge 1}T_u^{(l)}$ as  in (\ref{nde}). By  Theorem \ref{1e} (with $\# U_1=\Re(U_1)$) and Theorem \ref{2e},
it is enough for (\ref{adx2}) to prove that  for each $p\in (\frac{2d^2}{d^2+1},2)$,
\begin{equation}\label{m78}
\|\sup_{u\in U_1}|T_u^{(l)}f|\|_p\les 2^{-l\e_0}\|f\|_p
\end{equation}
holds for some $\e_0>0$.  In fact, by rescaling,  (\ref{m78}) follows from \begin{equation}\label{m79}
\big\|\sup_{k\in\Z}|\sum_{j\in\Z}\Xi_{j,l}^{2^lu_k}f|\big\|_p\les 2^{-l\e_0}\|f\|_p,
\end{equation}
where $\Xi_{j,l}^{2^lu_k}$ is given by (\ref{m77}) with $t=2^lu_k$. Using $u_k\in [2^{n(k)},2^{n(k)+1})$ and the multiplier of $\Xi_{j,l}^{2^lu_k}f$
\begin{equation}\label{m90}
\psi(u_k|\de_j^\A(\xi)|) \int e^{i2^lu_k\de_j^\A(\xi)\cdot \g(s)}\Upsilon\big(u_k \g'(s)\cdot \de_j^\A(\xi)\big)\rho(s)ds,
\end{equation}
where $\Upsilon$  is given as in (\ref{B}), we see that
 the support of $\F\{\Xi_{j,l}^{2^lu_k} f\}(\xi)$ is a subset of
\begin{equation*}
\bigcup_{s:|s|\in[1/2,2]}\Big\{\xi\in\R^d:\ 2^{-n(k)-2}\le |\de_j^\A(\xi)|\le 2^{-n(k)+2},\ \  |\g'(s)\cdot \de_j^\A(\xi)|\le 2^{-n(k)+2}c_0\Big\}
\end{equation*}
with $c_0$ small enough. We then obtain from   Lemma \ref{l9} with $k$ replaced by $n(k)$ that
\begin{equation}\label{m80}
\Big|\sum_{j\in\Z}\Xi_{j,l}^{2^lu_k}f\Big|(x)\les\  \sum_{l'=1}^d\sum_{l''\in \{1,\cdots,d\}\setminus\{l'\}}\Big|\sum_{j\in\Z}\Xi_{j,l}^{2^lu_k}
\PP_{j\A_{l'}-n(k)}^{(l')}
\bar{\PP}_{j\A_{l''}-n(k)}^{(l'')}(K_{l',l''}^{j,n(k)}*f)\Big|(x)
\end{equation}
where $\PP_{k}^{(n)}$ and $\bar{\PP}_{k}^{(n)}$ are  variants of the Littlewood-Paley operator $P_k^{(n)}$, and
\begin{equation}\label{m81}
|K_{l',l''}^{j,n(k)}*f|(x)\les M^{(1,d)}\circ M^{(1,d)}f(x).
\end{equation}
Plugging (\ref{m80}) into (\ref{m79}), we see by (\ref{m81}) and the Littlewood-Paley theory that it suffices to prove  that for all $l'\neq l''$, and for each $p\in (\frac{2d^2}{d^2+1},2)$,
 \begin{equation}\label{m83}
\|\big(\sum_{(k,j)\in\Z^2}|\Xi_{j,l}^{2^lu_k}\PP_{j\A_{l'}-n(k)}^{(l')}\PP_{j\A_{l''}-n(k)}^{(l'')}f|^2\big)^{1/2}\|_p
\les\ 2^{-l\e_0}\|f\|_p.
\end{equation}
  Let $\{r_i(\cdot)\}_{i=0}^\infty$ be the sequence of the Rademacher functions. We further reduce the proof of    (\ref{m83}) to showing that for $p\in (\frac{2d^2}{d^2+1},2)$,
 \begin{equation}\label{m84}
\|\sum_{(k,j)\in\Z^2}r_k(\tau_1)r_j(\tau_2)\Xi_{j,l}^{2^lu_k}\PP_{j\A_{l'}-n(k)}^{(l')}\PP_{j\A_{l''}-n(k)}^{(l'')}f\|_p
\les\ 2^{-l\e_0}\|f\|_p
\end{equation}
with the implicit constant independent of $\tau_1$ and $\tau_2$.
By the method of stationary phase,
we get from (\ref{m90}) and   Plancherel's theorem that
 \begin{equation}\label{m88}
\|\sum_{(k,j)\in\Z^2}r_k(\tau_1)r_j(\tau_2)\Xi_{j,l}^{2^lu_k}\PP_{j\A_{l'}-n(k)}^{(l')}\PP_{j\A_{l''}-n(k)}^{(l'')}f\|_2
\les\ 2^{-l/d}\|f\|_2.
\end{equation}
 Moreover, applying  Proposition  \ref{mtmt} to the multiplier
 $$\sum_{(k,j)\in\Z^2}r_k(\tau_1)r_j(\tau_2) \Psi(\frac{\xi_{l'}}{2^{j\A_{l'}-n(k)}})\tilde{\Psi}(\frac{\xi_{l''}}{2^{j\A_{l''}-n(k)}})
 \times\  (\ref{m90})$$
($\Psi$ and $\tilde{\Psi}$ are defined as in the proof of  Lemma \ref{l9}),
 which satisfies (\ref{n08}) with $C_H$ replaced by $2^{l(d-{1}/{d}+\eta_0)}$ ($\eta_0$ sufficiently small),
 we then deduce that for $p\in (1,\infty)$
 \begin{equation}\label{m89}
\|\sum_{(k,j)\in\Z^2}r_k(\tau_1)r_j(\tau_2)\Xi_{j,l}^{2^lu_k}\PP_{j\A_{l'}-n(k)}^{(l')}\PP_{j\A_{l''}-n(k)}^{(l'')}f\|_p
\les\ 2^{l(d-{1}/{d}+\eta_0)}\|f\|_p.
\end{equation}
Finally,
 interpolating  between (\ref{m88}) and (\ref{m89})   implies the desired
(\ref{m84}).
\end{proof}
\section{Lower bound in Theorem \ref{t1}}
\label{sec8}
 In this section, we prove the lower bound in Theorem \ref{t1}.
 \begin{thm}\label{t8.1}
 Let $U\subset(0,\infty)$ and $p\in(1,\infty)$. Then we have
 \begin{equation}\label{8.1bb}
 \|\h^U\|_{L^p\to L^p}\gtrsim  \sqrt{\log(e+\Re (U))},
 \end{equation}
where the implicit constant is independent of $U$.
 \end{thm}
 \begin{rem}\label{asum}
 In what follows, we may assume that $\Re(U)$ is sufficiently large, since Theorem \ref{t8.1} is a direct result of the lower bound for the Hilbert transform along a fixed curve.
 \end{rem}
 \subsection{Reduction  to  $\|\h^U\|_{L^2\to L^2}$}\label{z01}
 In this subsection, we  shall prove  (\ref{8.1bb}) for $p\neq 2$ under the assumption that (\ref{8.1bb})  holds for  $p=2$.
 We first introduce  a new set  $U_\circ$ by picking arbitrary  one  element in each set  $[2^n,2^{n+1})\cap U$ with $n\in S_0:=\{n\in\Z:\ [2^n,2^{n+1})\cap U\neq \emptyset\}$.
   Clearly, $\Re(U)=\#U_\circ$ and $U_\circ$  can be seen as a disjoint union of two 2-lacunary sets. So  we  can deduce
    by Proposition \ref{lac}
 \begin{equation}\label{v10}
  \|\h^{U_\circ}\|_{L^q\to L^q}\les \sqrt{\log(\#U_\circ)}
 \end{equation}
 whenever $q\in  (\frac{2d^2}{d^2+1},\infty)$.
 Note that for $p\in (1,\infty)$, there are $\theta\in (0,1)$ and $q\in (\frac{2d^2}{d^2+1},\infty)$ such that  $1/2=\theta/p+(1-\theta)/q$, which implies
 the   interpolation
$  \|\h^{U_\circ}\|_{L^2\to L^2}\le   \|\h^{U_\circ}\|_{L^p\to L^p}^{\theta}  \|\h^{U_\circ}\|_{L^q\to L^q}^{1-\theta}.$
Therefore, we can obtain (\ref{8.1bb}) by combining   (\ref{v10}) and the assumption
 \begin{equation}\label{v11}
  \|\h^{U_\circ}\|_{L^2\to L^2}\gtrsim  \sqrt{\log(\#U_\circ)}.
 \end{equation}
It  remains to prove (\ref{v11}), which is the goal in the following context. To enhance the clarity of the proof, we provide a brief outline of the proof.  To begin with, we introduce two approximations for the multiplier of the Hilbert transform $\h^{(1)}$ (i.e., $\h^{(u)}$ with $u=1$) in Subsection \ref{sub00}. Then, we reduce the proof of (\ref{v11}) to demonstrating Proposition \ref{pp30} in Subsection \ref{sub11}. Finally, in Subsection \ref{sub33}, we establish the validity of this proposition by utilizing Proposition \ref{pp40}.
\subsection{Approximations to the multiplier}
\label{sub00}
Before we treat the multiplier of the Hilbert transform, we need the following  lemma to provide   crucial  decay estimates for   oscillatory integrals
whose phases are certain  ``fractional" polynomials.
\begin{lemma}\label{lemxin}
Let $n\ge 1$,  $P(t)=t^{b_0}+\sum_{i=1}^n \mu_i t^{b_i}$ be a real-valued function on  $\R^+$, where $b_0,b_1,\cdots,b_n$ are  distinct positive exponents and $\mu_1,\cdots,\mu_n$ are arbitrary real parameters.  Let $a\ge 1$. Then
\begin{equation}\label{iex}
\big|\int_a^\infty e^{iP(t)}\frac{dt}{t} \big|\les_{n,b_i} a^{-\frac{b_0}{n+1}}.
\end{equation}
\end{lemma}
\begin{proof}
Making the change of variable $t\to t^{k_n}$ with $k_n=(n+1)/b_0$, we express the integral on the left-hand side of (\ref{iex}) as $I_n=k_n\int_{a^{1/k_n}}^\infty e^{iQ(t)}\frac{dt}{t}$ with $Q(t):=P(t^{k_n})$. Then we  reduce the matter  to proving
\begin{equation}\label{iex1}
|I_n|\les_{n,b_i} a^{-\frac{b_0}{n+1}}.
\end{equation}
Note that $Q''(t)$ has at most $n$ zeros on $(0,\infty)$ (see e.g., Lemma 2 in \cite{CFWZ08}), which are denoted by  $\{t_1,t_2,\cdots,t_{l}\}$ with $l\le n$.   In addition,   writing  $t_0:=a^{1/k_n}$ and $t_{l+1}:=\infty$,  we may rewrite $I_n$ as
 $$I_n=k_n\sum_{j=1}^{l+1}I_n^{(j)},\hskip.2in{\rm  where}\hskip.2in  I_n^{(j)}:=
 \int_{t_{j-1}}^{t_{j}}e^{iQ(t)}\frac{dt}{t}.$$
 Consequently,  (\ref{iex1}) follows from
  \begin{equation}\label{iex2}
|I_n^{(j)}|\les_{n,b_i} a^{-\frac{b_0}{n+1}},\ \  j=1,\cdots,l+1.
\end{equation}
It remains to prove (\ref{iex2}).
 Observe $k-k_nb_0\le 0$ for all $k=1,\cdots,n+1$. Applying Lemma 2.5 in \cite{LL07} to $Q(t)=t^{k_nb_0}+\mu_1t^{k_nb_1}+\cdots+\mu_nt^{k_nb_n}$,
 we get from  $ t>a^{1/k_n}$ that
$$1\les\ \sum_{k=1}^{n+1}t^{k-k_nb_0}|Q^{(k)}(t)|\les\  \sum_{k=1}^{n+1}a^{k/k_n-b_0}|Q^{(k)}(t)|.$$
 Moreover,
because $Q'(t)$ is monotonic on each $(t_{j-1},t_j)$,
 (\ref{iex2}) follows  from  van der Corput's  Lemma (see \cite{S93} Page 334).
\end{proof}
Next, we shall give two approximations  to  the multiplier of the Hilbert transform $\h^{(1)}$
$$m(\xi)={\rm p.v.}\int_{\R} e^{i(t\xi_1+t^{\A_2}\xi_2+\cdots+t^{\A_d}\xi_d)}\frac{dt}{t},$$
where $\{\A_i\}_{i=2}^d$ are given as in (\ref{alpha}).
We first  restrict  the variable  $\xi$ to a special unbounded set
\begin{equation}\label{aa21}
\mathfrak{U}:=\{\xi\in\R^d:\ \xi_1>0,\ \xi_2>0,\ |\xi_i|\le \big(|\xi_2|^\frac{1}{\A_2}|\xi_1|\big)^\frac{\A_i}{2}\ {\rm for}\ i=3,\cdots,d\},
\end{equation}
which is closely related to the  sets $\{S_j\}$  (see  (\ref{C}) below).
In particular,  if $\xi_1>0$ and $\xi_2>0$, we have
\begin{equation}\label{aa13}
m(\xi_1,0,\cdots,0)=\pi i\ \ {\rm and}\ \  m(0,\xi_2,0,\cdots,0)=\Xi_0(\A_2),
\end{equation}
where $\Xi_0(\A_2)$ equals
0 if $\A_2$ is even, and equals ${\pi i}/{\A_2}$ if $\A_2$ is odd.
\vskip.1in
We now  approximate $m(\xi)$ by  $m(\xi_1,0,\cdots,0)$ and $m(0,\xi_2,0,\cdots,0)$, respectively,    while  the  estimates of the resulting errors are essential  in the  proof of (\ref{v11}).
\begin{lemma}\label{z8}
Let $\mathfrak{U}$ be as in (\ref{aa21}), and let $b=\frac{1}{2d(\A_2+1)}$. Then there is a positive constant $C^*$ independent of  $\xi$ such that for all $\xi\in \mathfrak{U}$,
\begin{align}
|m(\xi)-m(\xi_1,0,\cdots,0)|\le&\  C^*\big(\frac{|\xi_2|}{|\xi_1|^{\A_2}}\big)^b\  {\rm and}\label{xi1}\\
|m(\xi)-m(0,\xi_2,0,\cdots,0)|\le&\  C^*\big(\frac{|\xi_1|^{\A_2}}{|\xi_2|}\big)^b\label{xi2}.
\end{align}
\end{lemma}
\begin{proof}
Since $|m(\xi)|\les1$,
one can easily get (\ref{xi1})  for ${|\xi_2|}\gtrsim {|\xi_1|^{\A_2}}$, and   (\ref{xi2}) for ${|\xi_2|}\les {|\xi_1|^{\A_2}}$.  So, in what follows, we only  prove (\ref{xi1}) and (\ref{xi2}) for
${|\xi_2|}\ll {|\xi_1|^{\A_2}}$ and  ${|\xi_2|}\gg{|\xi_1|^{\A_2}}$, respectively.
We first show  (\ref{xi1}) for ${|\xi_2|}\ll {|\xi_1|^{\A_2}}$.  Let $\e_i=\frac{\A_i}{2(\A_i+1)\A_2}$ for $i=3,\cdots,d$, and
$\eta_j={\xi_j}{|\xi_1|^{-\A_j}}$ for $j=2,\cdots,d$.
Since  $\xi\in\mathfrak{U}$, we have  $|\eta_2|\ll1$.
By the change of variable $t\to |\xi_1|^{-1}t$, we see from  $\xi_1>0$ that it suffices to establish the inequality
\begin{align}
|m(1,\eta_2,\cdots,\eta_d)-m(1,0,\cdots,0)|\les&\  |\eta_2|^b \label{xi3}
\end{align}
for $|\eta_i|\le |\eta_2|^{(\A_i+1)\e_i}$, $i=3,\cdots,d$.  Let $\e=db$, which is smaller than  $\min_{3\le i\le d}\e_i$ (since $\A_i>\A_2$ for $i\ge 3$),  and let
\begin{equation}\label{c01}
A_\e=|\eta_2|^{-\e}.
\end{equation}
The left-hand side of (\ref{xi3}) is bounded by the sum of $J_1$, $J_2$ and $J_3$,
which are given by
$$
\begin{aligned}
J_1=&\ |\int_{|t|\le A_\e} e^{it}
\big(e^{i(t^{\A_2}\eta_2+\cdots+t^{\A_d}\eta_d)}-1\big)\frac{dt}{t}|,\\
J_2=&\  |\int_{ A_\e}^\infty
e^{i(t+t^{\A_2}\eta_2+\cdots+t^{\A_d}\eta_d)}\frac{dt}{t}|\\
&+|\int^{- A_\e}_{-\infty}
e^{i(t+t^{\A_2}\eta_2+\cdots+t^{\A_d}\eta_d)}\frac{dt}{t}|,\\
J_3=&\ |\int_{ A_\e}^\infty
e^{it}\frac{dt}{t}|+|\int^{- A_\e}_{-\infty}
e^{it}\frac{dt}{t}|.
\end{aligned}$$
We first get by integrating     by parts that $J_3\les A_\e^{-1}$. Invoking $|\eta_i|\le |\eta_2|^{(\A_i+1)\e_i}$ for  $i=3,\cdots,d$, we then have by $\e(1+\A_2)<1$ and (\ref{c01}) that
$$
\begin{aligned}
J_1\les&\   \int_{|t|\le A_\e}
\big(|t|^{\A_2-1}|\eta_2|+\cdots+t^{\A_d-1}|\eta_d|\big)dt\\
\les&\ A_\e^{\A_2}|\eta_2|+\cdots+A_\e^{\A_d}|\eta_d|
\les |\eta_2|^\e.
\end{aligned}
$$
Thanks to Lemma \ref{lemxin}, the first absolute value in the expression of  $J_2$ is $\les A_\e^{-1/d}$.    By changing the variable $t\to -t$,  the second absolute value  in the expression of $J_2$   is $\les A_\e^{-1/d}$ as well. Collecting the above estimates of $J_1$, $J_2$ and $J_3$, we finally achieve (\ref{xi3}) from (\ref{c01}).

Next,
we  show  (\ref{xi2}) for  ${|\xi_2|}\gg {|\xi_1|^{\A_2}}$.  Let $\e_0=\A_2\e$,  $\tilde{\e}_i=\A_2 \e_i$  for $i=3,\cdots,d$ (thus $\min_{3\le i\le d}\tilde{\e}_i>\e_0$), and let
$\zeta_1=\xi_1|\xi_2|^{-{1}/{\A_2}}$,  $\zeta_j=\xi_j|\xi_2|^{-{\A_j}/{\A_2}}$ for $j=3,\cdots,d$. So $|\zeta_1|\ll1$. Since $\xi\in \mathfrak{U}$,
by the change of variable $t\to \xi_2^{-{1}/{\A_2}}t$,
it suffices to show that for all $|\zeta_i|\le |\zeta_1|^{(\A_i+1)\tilde{\e}_i}$  $(i=3,\cdots,d)$,
\begin{align}
|m(\zeta_1,1,\zeta_3,\cdots,\zeta_d)-m(0,1,0,\cdots,0)|\les&\  |\zeta_1|^{\e_0/d}\label{xi4}.
\end{align}
We can bound
the left-hand side of (\ref{xi4}) by
$\sum_{i=1,2,3}L_i$, where
$$
\begin{aligned}
L_1=&\ |\int_{|t|\le B_{\e_0}} e^{it^{\A_2}}
\big(e^{i(t\zeta_1+t^{\A_3}\zeta_3+\cdots+t^{\A_d}\zeta_d)}-1\big)\frac{dt}{t}|,\\
L_2=&\  |\int_{ B_{\e_0}}^\infty
e^{i(t\zeta_1+t^{\A_2}+t^{\A_3}\zeta_3+
\cdots+t^{\A_d}\zeta_d}\frac{dt}{t}|\\
&+|\int^{ -B_{\e_0}}_{-\infty}
e^{i(t\zeta_1+t^{\A_2}+t^{\A_3}\zeta_3+\cdots+t^{\A_d}\zeta_d)}\frac{dt}{t}|,\\
L_3=&\ |\int_{ B_{\e_0}}^\infty
e^{it^{\A_2}}\frac{dt}{t}|+|\int^{ -B_{\e_0}}_{-\infty}
e^{it^{\A_2}}\frac{dt}{t}|.
\end{aligned}$$
We first deduce   by integrating by parts that $L_3\les B_{\e_0}^{-\A_2}$, where
$
B_{\e_0}:=|\zeta_1|^{-\e_0}.
$
 Applying  $|\zeta_i|\le |\zeta_1|^{(\A_i+1)\tilde{\e}_i}$ for  $3\le i\le d$,  we deduce by $\e_0<1/2$ and $\e_0<\tilde{\e}_i$  ($3\le i\le d$)
that
$$
\begin{aligned}
L_1
\les&\ \int_{|t|\le B_{\e_0}} (|\zeta_1|+|t|^{\A_3-1}|\zeta_3|+\cdots+|t^{\A_d-1}||\zeta_d|)dt\\
\les&\  B_{\e_0}|\zeta_1|+B_{\e_0}^{\A_3}|\zeta_3|+\cdots+B_{\e_0}^{\A_d}|\zeta_d|
\les\ |\zeta_1|^{\e_0}.
\end{aligned}
$$
Arguing similarly as in the previous  estimate  of $J_2$, we can also obtain
$L_2\les |\zeta_1|^{\e_0\A_2/d}$ by  Lemma \ref{lemxin}.  We finally  conclude the proof of (\ref{xi4})  by combining the aforementioned estimates of $L_1$, $L_2$ and  $L_3$.
\end{proof}
\subsection{Reduction of  (\ref{v11})}
\label{sub11}
This subsection reduces the proof of (\ref{v11}) to proving  the  Proposition \ref{pp30} below.
Keep in mind that  $U_\circ\subset U$ with $\# U_\circ<\infty$. Denote by $S_0'$ the set $\{n\in\Z:\ [2^n,2^{n+1})\cap U_\circ\neq \emptyset\}$. From the definition of $U_\circ$ we deduce $\# S_0'=\#U_\circ$. As  in $Remark$ \ref{asum}, in what follows, we may assume
$\# U_\circ\ge 2^{10(C^*+\A_2)}$.
Define  the constant $K$  by
\begin{equation}\label{p10}
K=K(U_\circ):=(C^*\# U_\circ)^{1/b}
\end{equation}
where $b$ and $C^*$ are given as in  Lemma \ref{z8}. Since $\# U_\circ\ge 2^{10(C^*+\A_2)}$, there is an integer $M$ such that $M+1$ can be expressed as an element in  $2^{\Z^+}$,  and such that
$$\# U_\circ\in [M,2M).$$
Let $S_0''$ be a maximal subfamily of $S_0'$ with the condition that the gap of arbitrary  two different integers in $S_0'$ is bigger than $1+\log_2(4K^2)$. Next,
 we may pick a {\it decreasing} sequence $\{u_1>u_2>\cdots>u_M\}$ such that each $u_j$
belongs to $U_\circ$ and to exactly one interval $[2^n,2^{n+1})$ with $n\in S_0''$. This choice yields
\begin{equation}\label{p20}
  u_j/u_{j+1}\ge 8K^2,\ j=1,2,\cdots, M-1.
\end{equation}
Hence, to prove (\ref{v11}), it suffices to show the following proposition.
\begin{prop}\label{pp30}
Let $U_\circ$ be as in Subsection \ref{z01}, and let  $\{u_j\}_{j=1}^M$ be as in (\ref{p20}). Then there is a positive constant $c$ independent of $M$
 such that
 $$\|\sup_{1\le j\le M}|\h^{(u_j)}f|\|_2\ge c\sqrt{\log M}$$
 holds for some $f$ with $\|f\|_2=1$.
 \end{prop}
 \subsection{A modification of Karagulyan's theorem}
 We shall introduce a useful theorem generalizing    Karagulyan's main result in \cite{Ka07}, see also Proposition 8.5 in Guo-Roos-Seeger-Yung \cite{Guo20}. For $\mu\in \Z^+$, we denote by
 $$W_\mu=\{\emptyset\}\cup \bigcup_{l=1}^{\mu-1}\{0,1\}^l$$
the set of binary words of length at most $\mu-1$, and define by
 $\tau:W_\mu\to \{1,\cdots,2^{\mu}-1\}$  the bijection. This
 bijection $\tau$  satisfies  $\tau(\emptyset)=2^{\mu-1}$ and
 $$\tau(\omega)=\omega_12^{\mu-1}+\omega_22^{\mu-2}+\cdots+\omega_l2^{\mu-l}+2^{\mu-l-1}$$
 if
 $\omega=\omega_1\omega_2\cdots\omega_l$ for certain $l\in \{1,\cdots,\mu-1\}$ and for some
 $\omega_1,\cdots,\omega_l\in\{0,1\}$.
 Following  the proof of Proposition 8.5 in \cite{Guo20}, one can also obtain  similar result in higher dimensions.
 \begin{prop}\label{pp40}
 Let $\mu$ be any positive integer, $M=2^\mu-1$, and let $S_1,\cdots,S_M$ be pairwise disjoint subsets of the whole space $\R^d$, where $S_j$ contains balls of arbitrary large radii. Then there is an $L^2$ function $f$ on $\R^d$ that has an orthogonal
 decomposition
 \begin{equation}\label{z1}
 f=\sum_{\omega\in W_\mu}f_\omega,
 \end{equation}
 where the functions $\{f_\omega\}$ satisfy
  \begin{equation}\label{z2}
 {\rm supp}\widehat{f_\omega}\subset S_{\tau(\omega)}\ {\rm for\  all\ } \omega\in W_\mu,\hskip.3in
 \|f\|_2^2=\sum_{\omega\in W_\mu}\|f_\omega\|_2^2\le2\ \ \ {\rm and}
  \end{equation}
   \begin{equation}\label{z3}
 \big\|\sup_{1\le j\le M}|\sum_{\omega\in W_\mu:\tau(\omega)\ge j}f_\omega|\big\|_2
 \ge \frac{\sqrt{\mu}}{100}\|f\|_2.
   \end{equation}
 \end{prop}
 We next prove Proposition \ref{pp30} by accepting  Proposition \ref{pp40}.   Invoking the arguments in  $Remark$ \ref{asum}, we assume that $M$ is sufficiently large.
 \subsection{Proof of Proposition \ref{pp30}}
 \label{sub33}
Keep (\ref{aa13}) in mind, and
define a sequence of sets $\{S_j\}_{j=1}^M$ by
\begin{equation}\label{C}
S_j:=\big\{\xi\in \R_+^d:
 \frac{1}{2Ku_j^{1-\A_2}}<\frac{\xi_2}{\xi_1^{\A_2}}< \frac{1}{Ku_j^{1-\A_2}},\ \xi_i\le u_M^{\beta_i}\big(|\xi_2|^\frac{1}{\A_2}|\xi_1|\big)^\frac{\A_i}{2}
 {\rm for}\ \ i=3,\cdots,d\big\},
\end{equation}
 where $\beta_i=\frac{\A_i(\A_2+1)-2\A_2}{2\A_2}>0$. It is clear that $\{S_j\}_{j=1}^M$ satisfy
 all corresponding conditions in Proposition \ref{pp40}. Thus,  there is an $L^2$ function $f$ on $\R^d$ such that (\ref{z1})-(\ref{z3}) hold.
For $1\le j\le M$,  we  deduce from (\ref{z1}) that
$$
\begin{aligned}
&|\h^{(u_j)}f(x)-\Xi_0(\A_2)f(x)|\ge \
\big|\sum_{\omega\in W_\mu:\tau(\omega)\ge j}(\pi i-\Xi_0(\A_2))f_\omega(x) \big|\\
&\ \ -
\big|\sum_{\omega\in W_\mu:\tau(\omega)\ge j} (\h^{(u_j)}-\pi i)f_\omega(x)\big|
-\big|\sum_{\omega\in W_\mu:\tau(\omega)< j} (\h^{(u_j)}-\Xi_0(\A_2))f_\omega(x) \big|,
 \end{aligned}
 $$
 where $\Xi_0(\A_2)$ is defined by the statements below (\ref{aa13}).
 Since $|\Xi_0(\A_2)-\pi i|\ge (1-\A_2^{-1})\pi$,  we have
 $$
\begin{aligned}
\sup_{1\le j\le M}|\h^{(u_j)}f(x)-\Xi_0(\A_2)f(x)|\ge& \
 (1-\A_2^{-1})\pi\sup_{1\le j\le M}\big|\sum_{\omega\in W_\mu:\tau(\omega)\ge j}f_\omega(x) \big|\\
&\ -
\sup_{1\le j\le M}\big|\sum_{\omega\in W_\mu:\tau(\omega)\ge j} (\h^{(u_j)}-\pi i)f_\omega(x)\big|\\
&\ -\sup_{1\le j\le M}\big|\sum_{\omega\in W_\mu:\tau(\omega)< j} (\h^{(u_j)}-\Xi_0(\A_2))f_\omega(x)\big|\\
=:&\ L_1(x)+L_2(x)+L_3(x).
 \end{aligned}
 $$
It follows from  (\ref{z3}) that there exists a constant $c_1>0$ such that
 \begin{equation}\label{c4}
 \|L_1\|_2\ge c_1\sqrt{\log_2 M}\|f\|_2.
 \end{equation}
 We next
    bound  $L_2(x)$ and $L_3(x)$ in order.
 Note that (\ref{z2}) implies  supp$\widehat{f_\omega}\subset S_{\tau(\omega)}$. Since
 $u_{\tau(\omega)}\le u_j$ (because $\tau(\omega)\ge j$ in $L_2(x)$),    $\xi\in {\rm supp}\widehat{f_\omega}$ and $\A_2>1$, we obtain
 $$\frac{u_j\xi_2}{(u_j\xi_1)^{\A_2}}
 =\frac{\xi_2}{u_j^{\A_2-1}\xi_1^{\A_2}}\le \frac{\xi_2}{u_{\tau(\omega)}^{\A_2-1}\xi_1^{\A_2}}\le K^{-1}\ \ \ {\rm and}$$
 $$\xi_i\le u_M^{\beta_i}
 \Big(|\xi_2|^\frac{1}{\A_2}|\xi_1|\Big)^{{\A_i}/{2}}\le u_j^{\beta_i}\Big(|\xi_2|^\frac{1}{\A_2
 }|\xi_1|\Big)^{{\A_i}/{2}}
\ \ \ \ \ \ {\rm  for}\ \ i=3,\cdots,d.$$
  By (\ref{xi1}) and  (\ref{p10}), we can infer
 $$|m(u_j\xi)-m(u_j\xi_1,0,\cdots,0)|\le C^*K^{-b}\le M^{-1},$$
 which implies
 $\|(\h^{(u_j)}-\pi i)f_\omega\|_2\le M^{-1}\|f_\omega\|_2$ by  Plancherel's theorem and (\ref{aa13}).
 Using this bound and (\ref{z2}), we further deduce
  \begin{equation}\label{c5}
 \begin{aligned}
 \|L_2\|_2\le&\  M^{1/2}\sup_{1\le j\le M}\|\sum_{\omega\in W_\mu:\tau(\omega)\ge j}|(\h^{(u_j)}-\pi i)f_\omega|\|_2\\
 \le&\ M\sup_{1\le j\le M}\big(\sum_{\omega\in W_\mu:\tau(\omega)\ge j}\|(\h^{(u_j)}-\pi i)f_\omega\|_2^2\big)^{1/2}\\
 \le&\  (\sum_{\omega}\|f_\omega\|_2^2)^{1/2}\le 2\|f\|_2,
 \end{aligned}
 \end{equation}
 where we  used  $l^2\subset l^\infty$ and Fubini's  theorem in the first inequality, and applied the Cauchy-Schwartz inequality in the second inequality.
 We now bound $L_3(x)$. Since
  $u_{\tau(\omega)}>u_j$ (because $\tau(\omega)<j$) and $\xi\in {\rm supp}\widehat{f_\omega}$, we obtain
$$\frac{u_j\xi_2}{(u_j\xi_1)^{\A_2}}=\frac{\xi_2}{u_{\tau(\omega)}^{\A_2-1}\xi_1^{\A_2}}(\frac{u_{\tau(\omega)}}{u_j})^{\A_2-1}\ge (2 K)^{-1} (8K^2)^{(\A_2-1)}\ge K,$$
where we  used $\A_2\ge 2$ (in fact, we can also treat the case $\A_2\in (1,2)$ by modifying the above construction of $S_0''$ such that $u_j/u_{j+1}\ge K^\frac{2}{\A_2-1}$).
By (\ref{xi2}) and  (\ref{p10}), we  have
 $$|m(u_j\xi)-m(0,u_j\xi_2,0,\cdots,0)|\le C^*K^{-b}\le M^{-1},$$
which leads to
 $\|(\h^{(u_j)}-\Xi(\A_2))f_\omega\|_2\le M^{-1}\|f_\omega\|_2$ by Plancherel's identity and (\ref{aa13}).
 Performing a similar arguments yielding the desired bound of $
 \|L_2\|_2$, we can also  get
  \begin{equation}\label{c6}
 \|L_3\|_2\le 2\|f||_2.
 \end{equation}
 Finally,  it follows by combining (\ref{c4})-(\ref{c6}) that
 $$\|\sup_{1\le j\le M}|\h^{(u_j)}f-\Xi_0(\A_2) f|\|_2\ge \frac{c_1}{2}\sqrt{\log_2M} \|f\|_2,$$
  which  immediately yields
 $$\|\sup_{1\le j\le M}|\h^{(u_j)}f|\|_2\ge \frac{c_1}{4}\sqrt{\log_2M} \|f\|_2$$
 by setting $M$ large enough such that $ c_1\sqrt{\log_2M}>4\Xi_0(\A_2) $. This
finishes  the proof of Proposition \ref{pp30}.
\section*{Acknowledgements}
The author thanks the anonymous referees
for their careful reading and helpful comments. 
 This work was supported by the NSF of China 11901301, 12161077.


\vskip .2in
\begin{thebibliography}{99}


\bibitem{ADP21}
N. Accomazzo, F. Di Plinio, I. Parissis,  Singular integrals along lacunary directions in $\R^n$, {\it Adv. Math. \bf 380} (2021), Paper No. 107580, 21 pp.

\bibitem{BGHS}  D. Beltran, S. Guo, J. Hickman, A. Seeger, Sharp $L^p$ bounds for the helical maximal function, arXiv:2102.08272 [math.CA], Amer. J. Math. (to appear).


\bibitem{B86}
J. Bourgain,  Averages in the plane over convex curves and maximal operators, {\it J. Analyse. Math. \bf 47}
(1986),  69-85.



\bibitem{BD15}
 J. Bourgain,  C. Demeter,  The proof of the $l^2$ decoupling conjecture, {\it Ann. of Math.  \bf 182} (2015),  351-389.

\bibitem{CS92}
 A. Carbery,  A. Seeger,  $H^p$ and $L^p$-variants of multiparameter Calder\'{o}n-Zygmund theory, {\it Trans. Amer. Math. Soc. \bf 334}  (1992), 719-747.

 \bibitem{CS95}
 A. Carbery,  A. Seeger,  Homogeneous Fourier multipliers of Marcinkiewicz type, {\it Ark. Mat. \bf 33} (1995), 45-80.

\bibitem{CWW85}
 S.-Y.A. Chang,  J.M. Wilson,  T.H. Wolff,  Some weighted norm inequalities concerning the Schr\"{o}dinger operators, {\it Comment. Math. Helv. \bf 60} (1985),  217-246.


\bibitem{CFWZ08}
J. Chen, D. Fan, M. Wang, X. Zhu, $L^p$ bounds for oscillatory hyper-Hilbert transform along curves, {\it Proc. Amer. Math. Soc. \bf 136} (2008), 3145-3153.

\bibitem{D10}
C. Demeter,  Singular integrals along $N$ directions in $\R^2$, {\it Proc. Am. Math. Soc. \bf 138} (2010),  4433-4442.

\bibitem{DD14}
C. Demeter, F. Di Plinio,  Logarithmic $L^p$ bounds for maximal directional singular integrals in the plane, {\it  J. Geom. Anal. \bf 24} (2014), 375-416.

\bibitem{DP18}
F. Di Plinio, I.  Parissis, A sharp estimate for the Hilbert transform along finite order lacunary sets of directions, {\it Israel J. Math. \bf 227} (2018), 189-214.


\bibitem{DP20}
F. Di Plinio, I.  Parissis, On the maximal directional Hilbert transform in three dimensions,{ \it Int. Math. Res. Not. IMRN}  (2020),  4324-4356.



\bibitem{DGTZ18}
F. Di Plinio, S.  Guo, C. Thiele,  P. Zorin-Kranich,  Square functions for bi-Lipschitz maps and directional operators, {\it  J. Funct. Anal. \bf 275}  (2018), 2015-2058.




\bibitem{DR86}
J. Duoandikoetxea, J.L.  Rubio de Francia, Maximal and singular integral operators via Fourier transform estimates, {\it  Invent. Math. \bf 84} (1986),  541-561.



\bibitem{GHS06}
L. Grafakos, P.  Honz\'{i}ak, A. Seeger,  On maximal functions for Mikhlin-H\"{o}rmander multipliers, {\it Adv. Math. \bf 204}  (2006), 363-378.

\bibitem{GS19}
L. Grafakos, L.  Slav\'{i}akov\'{a},  The Marcinkiewicz multiplier theorem revisited, {\it Arch. Math. (Basel) \bf 112} (2019), 191-203.
\bibitem{GG14}
L. Grafakos, Classical Fourier Analysis, third edition, Graduate Texts in Mathematics, vol. 249, Springer, New York, 2014.

\bibitem{G09}
P.T. Gressman, $L^p$-improving properties of averages on polynomial curves and related integral estimates, { \it  Math. Res. Lett. \bf 16}  (2009) 971-989.

\bibitem{GHLR}
S. Guo, J. Hickman, V. Lie, J. Roos,  Maximal operators and Hilbert transforms along variable non-flat homogeneous curves, {\it  Proc. Lond. Math. Soc. \bf 115} (2017),  177-219.


\bibitem{Guo20} S. Guo, J. Roos, A. Seeger, P.-L. Yung,  A maximal function for families of Hilbert transforms along homogeneous curves, {\it Math. Ann. \bf 377} (2020), 69-114.



\bibitem{GRSY20}
S. Guo, J. Roos, A. Seeger, P.-L. Yung, Maximal functions associated with families of homogeneous curves: $L^p$ bounds for $p\le 2$,  {\it Proc. Edinb. Math. Soc. \bf 63} (2020),  398-412.


\bibitem{Guth}
L. Guth,
Decoupling estimates in Fourier analysis, arXiv:2207.00652 [math.CA].






\bibitem{G16}
L. Guth,  A restriction estimate using polynomial partitioning, {\it J. Amer. Math. Soc., \bf 29} (2016)  371-413.


\bibitem{HKL20}
S. Ham, H. Ko, S. Lee, Remarks on estimates for the adjoint restriction operator to curves over the sphere, {\it J. Fourier Anal. Appl. \bf 26} (2020),  Paper No. 2.



\bibitem{HL14}
S. Ham, S. Lee,  Restriction estimates for space curves with respect to general measures, {\it  Adv. Math. \bf 254}  (2014), 251-279.



\bibitem{H16}
 J. Hickman,  Uniform $L^p_x$-$L^q_{x,r}$ improving for dilated averages over polynomial curves, {\it  J. Funct. Anal. \bf 270}  (2016),  560-608.



\bibitem{Ka07}
G.A. Karagulyan, On unboundedness of maximal operators for directional Hilbert transforms, {\it  Proc. Am. Math. Soc. \bf 135} (2007), 3133-3141.



\bibitem{KP22}
J. Kim, M.  Pramanik, $L^2$ bounds for a maximal directional Hilbert transform, {\it Anal. PDE \bf 15} (2022), 753-794.



\bibitem{KLO22}  H. Ko, S. Lee, S. Oh,   Maximal estimates for averages over space curves, {\it  Invent. Math.  \bf 228}  (2022),  991-1035.

\bibitem{KLO23} H. Ko, S. Lee, S. Oh,   Sharp smoothing properties of averages over curves,
{\it  Forum Math. Pi  \bf 11} (2023),  Paper No. e4, 33 pp.



\bibitem{LMP19}
I. \L aba, A. Marinelli, M. Pramanik, On the maximal directional Hilbert transform, {\it Anal. Math. \bf 45} (2019), 535-568.

\bibitem{LL07}
N. Laghi, N. Lyall,  Strongly singular integrals along curves, {\it  Pacific J. Math. \bf 233}  (2007),  403-415.

\bibitem{Lee03}
S. Lee,  Endpoint estimates for the circular maximal function, {\it Proc. Amer. Math. Soc. \bf 131} (2003),  1433-1442.




  \bibitem{LY21}
J. Li, H. Yu, $L^2$ boundedness of Hilbert transforms along variable flat curves, {\it  Math. Z. \bf 298} (2021),  1573-1591.

\bibitem{L20}
V.  Lie,  The polynomial Carleson operator, {\it Ann. of Math. (2) \bf 192} (2020),  47-163.



\bibitem{L19} V. Lie,  A unified approach to three themes in harmonic analysis (1st part), arXiv:1902.03807.




\bibitem{LSY21}
N. Liu, L. Song, H. Yu, $L^p$ bounds of maximal operators along variable planar curves in the Lipschitz regularity, {\it J. Funct. Anal. \bf 280} (2021), Paper No. 108888, 40 pp.



\bibitem{PS07}
M. Pramanik, A. Seeger,   $L^p$ regularity of averages over curves and bounds for associated maximal operators, {\it  Amer. J. Math. \bf 129} (2007) 61-103.



\bibitem{Sc97}
W. Schlag, A generalization of Bourgain's circular maximal theorem, {\it J. Am. Math. Soc. \bf 10} (1997) 103-122.

\bibitem{Sc98}
W. Schlag,  A geometric proof of the circular maximal theorem, {\it Duke Math. J. \bf 93}  (1998) 505-533.

\bibitem{Se88}
A. Seeger, Some inequalities for singular convolution operators in $l^p$-spaces, {\it Trans. Am. Math. Soc. \bf 308} (1988), 259-272.



\bibitem{S93}
E. Stein, Harmonic analysis: real-variable methods, orthogonality, and oscillatory integrals. With the assistance of Timothy S. Murphy, Princeton Mathematical Series 43, Monographs in Harmonic Analysis, III (Princeton University Press, Princeton, NJ, 1993) xiv+695.


\bibitem{S76}
E.M. Stein,  Maximal functions: spherical means, {\it  Proc. Nat. Acad. Sci. USA \bf 73}   (1976), 2174-2175.


\bibitem{S87}
E.M. Stein, Problems in harmonic analysis related to curvature and oscillatory integrals, Proceedings of the International Congress of Mathematicians, Vol. 1, 2 (Berkeley, Calif., 1986), 196-221, Amer. Math. Soc., Providence, RI, 1987.



\bibitem{SW78}
E.M. Stein, S. Wainger, Problems in harmonic analysis related to curvature, {\it Bull. Am. Math. Soc. \bf 84} (1978), 1239-1295.

 \bibitem{SW01}
 E. M. Stein and S. Wainger, Oscillatory integrals related to Carleson's theorem, {\it Math. Res. Lett. \bf 8}  (2001), 789-800.

\bibitem{S14}
B. Stovall, Uniform $L^p$-improving for weighted averages on curves,  {\it Anal. PDE  \bf 7}  (2014) 1109-1136.


\bibitem{TW03}
T. Tao,  J. Wright,  $L^p$ improving bounds for averages along curves, {\it J. Amer. Math. Soc. \bf 16}  (2003)   605-638.



\bibitem{Wan19} R. Wan,  $L^p$ bound for the Hilbert transform along variable non-flat curves,
{\it Math. Nachr. \bf 296}  (2023)  1669-1686.


\bibitem{Wan22}
R. Wan,
Uniform $L^p$ estimates for  Hilbert transform and maximal operator along a  class of variable curves, arXiv:2209.05825v2 [math.CA].






 \bibitem{W00}
T. Wolff, Local smoothing type estimates on $L^p$ for large $p$, {\it Geom. Funct. Anal. \bf 10} (2000), 1237-1288.



  \bibitem{ZK21}
  P. Zorin-Kranich,   Maximal polynomial modulations of singular integrals, {\it  Adv. Math. \bf 386}  (2021), Paper No. 107832, 40 pp.



















\end{thebibliography}
\end{document}